\documentclass[11pt, draft]{amsart}
\usepackage{amssymb, amstext, amscd, amsmath, amssymb}
\usepackage{mathtools, color, paralist, dsfont, rotating}
\usepackage{verbatim}
\usepackage{enumerate}
\usepackage{tikz-cd}
\usepackage{tikz}
\usepackage[all]{xy}
\numberwithin{equation}{section}

\makeatletter
\def\@cite#1#2{{\m@th\upshape\bfseries%
[{#1\if@tempswa{\m@th\upshape\mdseries, #2}\fi}]}}
\makeatother
\theoremstyle{plain}
\newtheorem{theorem}{Theorem}[section]
\newtheorem{corollary}[theorem]{Corollary}
\newtheorem{proposition}[theorem]{Proposition}
\newtheorem{lemma}[theorem]{Lemma}
\theoremstyle{definition}
\newtheorem{definition}[theorem]{Definition}
\newtheorem{example}[theorem]{Example}

\newtheorem{remark}[theorem]{Remark}

\theoremstyle{remark}

\mathtoolsset{centercolon}
  \newcommand{\cA}{{\mathcal{A}}}
  \newcommand{\cB}{{\mathcal{B}}}
  \newcommand{\cC}{{\mathcal{C}}}

  \newcommand{\cF}{{\mathcal{F}}}
  
	\newcommand{\cH}{{\mathcal{H}}}
  
  \newcommand{\cJ}{{\mathcal{J}}}
  \newcommand{\cK}{{\mathcal{K}}}
	\newcommand{\cL}{{\mathcal{L}}}
  \newcommand{\cM}{{\mathcal{M}}}
  \newcommand{\cN}{{\mathcal{N}}}
	\newcommand{\cO}{{\mathcal{O}}}

  \newcommand{\cT}{{\mathcal{T}}}

\newcommand{\A}{{\mathcal{A}}}
\newcommand{\B}{{\mathcal{B}}}
\newcommand{\C}{{\mathcal{C}}}

\newcommand{\F}{{\mathcal{F}}}
\newcommand{\G}{{\mathcal{G}}}
\renewcommand{\H}{{\mathcal{H}}}

\newcommand{\J}{{\mathcal{J}}}
\newcommand{\K}{{\mathcal{K}}}
\renewcommand{\L}{{\mathcal{L}}}

\newcommand{\N}{{\mathcal{N}}}
\renewcommand{\O}{{\mathcal{O}}}

\newcommand{\T}{{\mathcal{T}}}

\newcommand{\bC}{\mathbb{C}}

\newcommand{\bN}{\mathbb{N}}

\newcommand{\bZ}{\mathbb{Z}}

\newcommand{\bbC}{{\mathbb{C}}}

\newcommand{\bbF}{{\mathbb{F}}}

\newcommand{\bbN}{{\mathbb{N}}}
\newcommand{\bbQ}{{\mathbb{Q}}}

\newcommand{\bbZ}{{\mathbb{Z}}}

\newcommand{\lip}{\langle}
\newcommand{\rip}{\rangle}

\newcommand{\sot}{\textsc{sot}}

\newcommand{\cenv}{\mathrm{C}_e^*}

\newcommand{\Aut}{\operatorname{Aut}}
\newcommand{\alg}{\operatorname{alg}}

\newcommand{\End}{\operatorname{End}}

\newcommand{\id}{id}

\newcommand{\Ind}{{\operatorname{Ind}}}

\newcommand{\Span}{\operatorname{span}}
\newcommand{\Ker}{\operatorname{Ker}}
\newcommand\cpr{\rtimes_{\alpha}^{r}\,{\mathcal{G}}}

\newcommand{\ca}{\mathrm{C}^*}
\begin{document}

\title{Tensor algebras of product systems and their $\ca$-envelopes}

\author[A. Dor-On]{Adam Dor-On}
\address{Department of Mathematics\\ University of Illinois at Urbana-Champaign\\ Urbana\\ IL 61801\\ USA}
\email{adoron@illinois.edu}

\author[E. Katsoulis]{Elias Katsoulis}
\address{Department of Mathematics \\ East Carolina University \\ Greenville \\ NC 27858-4353 \\ USA}
\email{katsoulise@ecu.edu}

\subjclass[2010]{Primary: 46L08, 46L55, 47B49, 47L40, 47L65, 46L05}
\keywords{Nica Toeplitz, Cuntz Nica Pimsner, C*-envelope, product systems, dilation, Hao-Ng isomorphism, quasi-lattice ordered semigroups, higher rank graphs}

\thanks{The first author was partially supported by an Ontario Trillium Scholarship, and an Azrieli international fellowship}

\maketitle

\begin{abstract}
Let $(G, P)$ be an abelian, lattice ordered group and let $X$ be a compactly aligned product system over $P$ with coefficients in $\A$. We show that the C*-envelope of the Nica tensor algebra $\N \T^+_X$ coincides with both Sehnem's covariance algebra $\A \times_X P$ and the co-universal $\ca$-algebra $\N\O^r_X$ for injective, gauge-compatible, Nica-covariant representations of Carlsen, Larsen, Sims and Vittadello. We give several applications of this result on both the selfadjoint and non-selfadjoint operator algebra theory. First we guarantee the existence of $\N\O^r_X$, thus settling a problem of Carlsen, Larsen, Sims and Vittadello which was open even for abelian, lattice ordered groups. As a second application, we resolve a problem posed by Skalski and Zacharias on dilating isometric representations of product systems to unitary representations. As a third application we characterize the $\ca$-envelope of the tensor algebra of a finitely aligned higher-rank graph which also holds for topological higher-rank graphs. As a final application we prove reduced Hao-Ng isomorphisms for generalized gauge actions of discrete groups on $\ca$-algebras of product systems. This generalizes recent results that were obtained by various authors in the case where $(G, P) =(\bbZ,\bN)$.
\end{abstract}

\section{Introduction}

Since its inception by Arveson \cite{Arv69} in the late 60's, the concept of the C*-envelope, or non-commutative Shilov boundary, has played an important role in operator algebra theory. It is through the ideas of non-commutative boundary theory that the work of Kennedy and Kalantar on C*-simplicity took flight \cite{KK17}, and in \cite{KS15} it is shown that the Arveson-Douglas conjecture can be phrased in terms of the C*-envelope of non-selfadjoint operator algebras arising from homogeneous polynomial relations.

The central result of this paper, Theorem \ref{T:CNP-envelope}, characterizes the $\ca$-envel\-ope of the Nica tensor algebra of a product system $X$ over $P$, where $(G,P)$ is an abelian, lattice ordered group. This is a new development even in the case of the abelian ordered groups $(\bbZ^d, \bbN^d)$ for $d \in \bbN$. A result in the spirit of Theorem \ref{T:CNP-envelope} has been coveted since the special case of the ordered group $(\bZ,\bN)$ was established in \cite{KK06b}, over a decade ago. Concrete cases have already appeared in the literature with the proofs sometimes requiring considerable effort (see for instance \cite[Subsection 4.3]{DFK14}). Our approach unifies these special cases and provides many new cases which were previously inaccessible by any other means. Even though Theorem \ref{T:CNP-envelope} is a non-selfadjoint algebra result, it has significant applications on $\ca$-algebra theory and a good part of the paper is devoted to them. Because of this, we view Theorem \ref{T:CNP-envelope} and the present work overall, as being right at the crossroads of the selfadjoint and non-selfadjoint operator algebra theory that should appeal to a wide audience. A more detailed description now follows.

In \cite{Pim97}, Pimsner generalized many constructions of operator algebras by associating them to C*-correspondences. Pimsner associates two C*-algebras $\cT_X$ and $\cO_X$ to a C*-correspondence $X$, where the algebra $\cO_X$ generalizes Cuntz-Krieger algebras \cite{CK80} arising from directed graphs and crossed products by $\mathbb{Z}$, while $\cT_X$ generalizes their Toeplitz extensions. In a sequence of papers \cite{Kat04a, Kat04b,Kat07}, Katsura drew insight from these specific examples and expanded Pimsner's construction. Katsura removed any assumptions on the C*-correspondence while simplifying and expanding many of Pimsner's results.

Many examples of non-selfadjoint operator algebras arise as tensor algebras of C*-correspondences, i.e., subalgebras of $\T_X$ generated by the copies of the $\ca$-correspondence and the coefficient algebra. Following Arveson's programme on the $\ca$-envelope, it makes sense to ask for a characterization of the $\ca$-envelope of a tensor algebra $\T^+_X$. Muhly and Solel \cite{MS98b} and Fowler, Muhly and Raeburn~\cite{FMR} established that the C*-envelope of $\cT^+_X$ is Pimsner's Cuntz-Pimsner algebra $\cO_X$, under various restrictions on the $\ca$-correspondence $X$. The problem for a general C*-correspondence was finally settled by Katsoulis and Kribs \cite{KK06b} who removed all restrictions and showed that the C*-envelope of any tensor algebra $\cT^+_X$ is Katsura's Cuntz-Pimsner algebra $\cO_X$. The proof required an intricate tail adding technique, as developed by Muhly and Tomforde in~\cite{MT04}. Although successful attempts have been made for C*-dynamical systems over $\bbN^d$ \cite{DFK14,Kak15}, this tail-adding technique is difficult to generalize beyond single $\ca$-correspondences. This has been an impediment for the development of non-selfadjoint algebra and C*-envelope theories. 

A product system of C*-correspondences over a partially ordered discrete group is a context which can be used to generalize the above constructs associated to a single $\ca$-correspondence. The development of this theory owes to the work of many hands \cite {Din91, LRaeb, Nic92, FR98, Fow99}. In order to obtain a satisfactory theory one needs to add the extra requirement that the product system is compactly aligned over a quasi-lattice ordered semigroup. Given a compactly aligned product system $X$ over a quasi-lattice ordered group with coefficients algebra $\A$, one builds a Nica-Toeplitz algebra $\N \T_X$ as a universal object for a suitable class of representations called \emph{Nica-covariant} representations. The theory of Nica-Toeplitz-Pimsner algebras $\N \T_X$ associated with product system $X$ reached its current state with the work of Fowler \cite{Fow02}, who successfully refined all previous ideas and provided a uniqueness theorem for it (see Theorem \ref{T:toep-uniq}). However, it was unclear what the right analogue of the Nica-type Cuntz-Pimsner algebra of a non-injective product system $X$ should be and this was left open for quite some time. Nevertheless this has now been completely settled by Sehnem's covariance algebra $\A\times_X P$ \cite{Seh+}. For a compactly aligned product system $X$ with coefficient algebra $\A$, Sehnem's covariance algebra $\A\times_X P$ plays the role of a Cuntz-Nica-Pimsner algebra, as it is a quotient of $\N\T_X$ containing a faithful copy of $X$ and satisfying a gauge invariance uniqueness theorem. Sehnem's algebra also incorporates earlier successful candidates for a Cuntz-Nica-Pimsner algebra, including the Sims-Yeend algebra $\N\O_X$ \cite{SY11} that appeared quite extensively in earlier versions of this paper.

Tensor algebras of compactly aligned product systems over abelian, lattice ordered groups provide a vast class of examples of non-selfadjoint operator algebras, and many natural operator algebras in the literature can be described as such. The list includes examples that do not materialize as tensor algebras of a single $\ca$-correspondence. These examples include the tensor algebras of $k$-graphs, studied by Davidson, Power and Yang \cite{DPY10} and the recent Nica semicrossed products of C*-dynamical systems over $\bN^d$, studied by Davidson, Fuller and Kakariadis \cite{DFK14} (see \cite{DFK14-2} for a survey). In these cases, the $\ca$-envelope has been calculated successfully, sometimes with considerable effort. Motivated by the special case of the tensor algebra of a $\ca$-correspondence, it is tempting to ask whether the aforementioned result of Katsoulis and Kribs~\cite{KK06b} holds in the greater generality of product systems, i.e., whether the $\ca$-envelope of the Nica tensor algebra of a product system is isomorphic to some Cuntz-Nica-Pimsner algebra of the system. 

Our Theorem \ref{T:CNP-envelope} answers this question by showing that if $X$ is a compactly aligned product system over an abelian, lattice ordered semigroup pair $(G,P)$, then indeed the $\ca$-envelope of the Nica tensor algebra $\cN\cT^+_X$ coincides with Sehnem's covariance algebra $\A \times_X P$. This result generalizes the earlier result of Katsoulis and Kribs~\cite{KK06b}, without the use of any tail adding technique in its proof (such a technique is currently unavailable for product systems over abelian, lattice-ordered group). 

One of the key ingredients in the proof of Theorem \ref{T:CNP-envelope} is the existence of a canonical isomorphism between the C*-envelope $\cenv(\N\T^+_X)$ and the co-universal $\ca$-algebra $\N\O^r_X$ for injective, gauge-compatible, Nica-covariant representations of Carlsen, Larsen, Sims and Vittadello \cite{CLSV11}. Not only do we list this fact in the statement of Theorem \ref{T:CNP-envelope} but we also consider it as its first application since it resolves an open problem in $\ca$-algebra theory. Indeed, the quest for co-universal algebras of this type originates with the work of Katsura in \cite[Proposition 7.14]{Kat07} for a single C*-correspondence. In \cite{CLSV11} the existence of such co-universal algebras is obtained for more general product systems under the additional assumption of $\widetilde{\phi}$-injectivity. The work in \cite{CLSV11} leaves open the case for various product systems, including product systems over $(\bbQ^d, \bbQ_+^d)$ and those of \cite[Example 3.16]{SY11} which are not $\tilde{\phi}$-injective. The C*-envelope approach shows that $\N\O_X^r$ always exists for product systems over abelian, lattice ordered groups. Furthermore, $\N\O_X^r$ can now be fully materialized within Arveson's programme and without any reference to gauge actions. This was desired (but not obtained) in \cite{CLSV11} and conjectured indirectly in \cite{DFK14-2}, as it was obtained there for a special (but nevertheless important) class of product systems over $(\bbZ^d, \bbN^d)$, $d \in \bbN$. (See \cite[Corollary  4.3.18]{DFK14-2}.)

We continue with more applications of our main result. In the context of a sourceless row-finite higher rank graph $\Lambda$, it was shown in \cite[Theorem 3.6]{KK06a} (see also \cite[Theorem 3.5]{DPY10}) that the C*-envelope of the tensor algebra $\cT_+(\Lambda)$ is the higher rank graph C*-algebra $\ca(\Lambda)$ of Kumjian and Pask \cite{KP00}. In \cite{RSY04} a more general C*-algebra was associated to a finitely aligned higher rank graph $\Lambda$, and a gauge invariant uniqueness theorem was proven for it. A tensor algebra $\cT_+(\Lambda)$ can still be defined, and our main theorem (Theorem \ref{T:CNP-envelope}) is used to easily extend \cite[Theorem 3.6]{KK06a} to finitely aligned higher rank graphs (see Theorem \ref{C:fin-aligned}).

In Section \ref{S:SZ} we provide another application of our results by resolving a question of Skalski and Zacharias from the end of \cite{SZ08}. These authors ask the following question: given a product system $X$ over $(\bZ^d,\bN^d)$, and a representation of $X$ in the sense that $(\sigma,T_1,...,T_d)$ is a $(d+1)$-tuple where each $(\sigma,T_i)$ is an isometric representation of $X_i$ with some commutation condition, is there a common dilation $(\pi,U_1,...,U_d)$ to a representation of $X$ where each $(\pi,U_i)$ is both isometric \emph{and} fully-coisometric. In Corollary \ref{C:unit-dil-SZ} we show that there is a positive answer to their question when $X$ is regular, and the representation $(\sigma,T_1,...,T_d)$ is \emph{doubly-commuting} (which is equivalent to Nica covariance in this case). We also construct in Example \ref{E:Nica-needed} a higher rank graph and an isometric representation for its associated higher-rank graph product system that has no isometric and fully coisometric dilation. This exhibits that Nica-covariance is necessary, and together with earlier observations in \cite{SZ08}, this shows that our result is optimal.

In Section \ref{S:HNI} we give our final application, this time to $\ca$-algebra theory. We use our main result to obtain analogues of the Hao-Ng isomorphism theorem in the context of product systems over more general semigroups. Recall that the Hao-Ng isomorphism problem, as popularized in \cite{BKQR}, asks for the validity of the isomorphism 
\begin{equation} \label{eq:HN}
\cO_X \cpr \cong \cO_{X \cpr}
\end{equation}
in the case where $X$ is a single $\ca$-correspondence and $\G$ a locally compact group acting on $X$. Even though this problem is still open in general, two important cases have been worked out: when $\G$ is amenable, by Hao and Ng in their original work~\cite{HN}, and more recently when $\G$ is discrete, by the second author~\cite{KatsIMRN}. In this paper we generalize the discrete reduced Hao-Ng isomorphism to compactly aligned product systems over abelian, lattice ordered groups. This generalizes results from \cite{BKQR, HN, KatsIMRN}.

\vspace{0.1in}

\section{Preliminaries}

\subsection{Operator algebras and $\ca$-envelopes}

We survey non-commutative boundary theory for unital operator algebras, and refer the reader to \cite{Arv69, Arv72, Arv+, BLM04} for a more in-depth treatment of the theory.

Let $\cA$ be an operator algebra. We say that the pair $(\B, \iota)$ is a {\em C*-cover} for $\cA$, if $\B$ is a C*-algebra, $\iota :\cA \rightarrow \cB$ is a completely isometric homomorphism, and $\ca(\iota(\cA)) = \cB$.

There is always a unique, smallest C*-cover for an operator algebra $\cA$. This C*-cover $(C_{e}^*(\cA), \kappa)$ is called the {\em C*-envelope} of $\cA$ and it satisfies the following universal property: given any other C*-cover $(\cB,\iota)$ for $\cA$, there exists a (necessarily unique and surjective) $*$-homomorphism $\pi:\cB \rightarrow C_{e}^*(\cA)$, such that $\pi \circ \iota = \kappa$. We will sometimes identify $\cA$ with its image $\iota(\cA)$ under a given C*-cover $(\B, \iota)$ for $\cA$. We say that $\rho : \A \rightarrow B(\H)$ is a \emph{representation} of $\A$ if $\rho$ is a completely contractive homomorphism.

The existence of the $\ca$-envelope for a \textit{unital} operator algebra was first established in the 70's by Hamana~\cite{Hamana}, following the pioneering work of Arveson~\cite{Arv69}. For a non-unital operator algebra, the existence of the $\ca$-envelope was established much later via unitization, which we now describe.

If $\cA \subseteq B(\cH)$ is a non-unital operator algebra generating a C*-algebra $\cB = \ca(\cA)$, a theorem of Meyer \cite[Section 3]{Mey01} (see also \cite[Corollary 2.1.15]{BLM04}) states that every representation $\varphi: \cA \rightarrow B(\cK)$ extends to a unital representation $\varphi^1$ on the \emph{unitization} $\cA^1=\cA \oplus \bC I_{\cH}$ of $\cA$ by specifying $\varphi^1(a+\lambda I_{\cH}) = \varphi(a) + \lambda I_{\cK}$. Meyer's theorem shows that $\cA$ has a \emph{unique} (one-point) unitization, in the sense that if $(\B, \iota)$ is a C*-cover for the operator algebra $\cA$, and $\cB \subseteq B(\cH)$ is some faithful representation of $\cB$, then the operator-algebraic structure on $\cA^1 \cong \iota(\cA) + \bC 1_{\cH}$ is independent of the C*-cover and the faithful representation of $\cB$.

The C*-envelope of a non-unital operator algebra can be computed from the C*-envelope of its unitization. More precisely, as the pair $(\ca_e(\cA), \iota)$ where $\ca_e(\cA)$ is the C*-subalgebra generated by $\iota(\cA)$ inside the C*-envelope $(\ca_e(\cA^1), \iota)$ of the (unique) unitization $\cA^1$ of $\cA$. By the proof of \cite[Proposition 4.3.5]{BLM04} this C*-envelope of an operator algebra $\cA$ has the desired universal property, that for any C*-cover $(\iota', \cB')$ of $\cA$, there exists a $*$-homomorphism $\pi:\cB' \to \ca_e(\cA)$, such that $\pi \circ \iota' = \iota$.
 
For a not necessary unital operator algebra $\cA$ generating a C*-algebra $\cB$, an ideal $\cJ$ of $\cB$ is called {\em a boundary ideal} for $\cA$ if the quotient map $\cB \to \cB/\cJ$ is a complete isometry on $\cA$. The existence of a $\ca$-envelope for $\A$ implies the existence of a largest boundary ideal $\cJ_{\cA}$ of $\cA$ in $\cB$, which is called {\em the Shilov ideal} of $\cA$ in $\cB$. Its importance in our context is that it gives a way of computing the C*-envelope. Namely, the C*-envelope of $\cA$ is always isomorphic to $\cB / \cJ_{\cA}$.

\subsection{Quasi-lattice ordered groups}

Let $P$ be a generating subsemigroup of a group $G$ such that $P\cap P^{-1} = \{e\}$, where $e$ is the identity element of $G$. Then $P$ induces a partial order on $G$ by defining $s \leq t$ if and only if $s^{-1}t \in P$ which is left-invariant in the sense that if $s\leq t$ then $rs \leq rt$ for $r,s,t \in G$. Following Nica \cite{Nic92}, we say that $(G,P)$ is a \emph{quasi-lattice ordered group} if every finite subset of $G$ with a common upper bound in $P$, has a (necessarily unique) least upper bound in $P$. When $s,t \in G$ have a common upper bound, we will denote their least upper bound by $s\vee t$; and when they do not, we write $s\vee t = \infty$. We note that by \cite[Lemma 7]{CL02}, the pair $(G,P)$ is quasi-lattice ordered if and only every finite subset of $G$ with a common upper bound in $G$ has a least upper bound in $G$. When any finite set in $G$ has both a least upper bound and greatest lower bound, we say that $(G,P)$ is lattice ordered.

Assume now that $(G,P)$ be a quasi-lattice ordered group with $G$ abelian. In that case we have that $G= PP^{-1}$, so that for any two elements $s$ and $t$ of $G$ we have $s\vee t<\infty$; indeed if $s=qp^{-1}$ and $t=q_1p_1^{-1}$ with $p,p_1, q, q_1\in P$, then $qq_1$ is an upper bound for $s$ and $t$.  Furthermore, we may define $s \wedge t= (s^{-1} \vee t^{-1})^{-1}$, $s,t\in G$. Equipped with the operation $\wedge$, it is easy to see that \textit{the abelian quasi-lattice ordered group $(G,P)$ becomes an abelian, lattice ordered group}.  Abelian, lattice ordered groups and product systems associated with them are the main focus of this paper.

\subsection{C*-correspondences}

Here we will give an overview of Hilbert C*-correspondences. For further details and material, we recommend \cite{Lan95}.

Let $\cA$ be a C*-algebra. A right inner product $\cA$-module is a complex vector space $X$ equipped with a right action of $\cA$ and an $\cA$-valued map $\lip \cdot , \cdot \rip : X\times X \rightarrow \cA$  which is $\cA$-linear in the second argument, such that for $x,y \in X$ and $a\in \cA$ we have 
\begin{enumerate}
\item
$\lip x,x\rip \geq 0$
\item
$\lip x,x \rip = 0$ if and only if $x=0$
\item
$\lip x,y \rip = \lip y,x \rip^*$.
\end{enumerate}
When $X$ is complete with respect to the norm given by $\| x \| = \| \lip x,x \rip \|^{\frac{1}{2}}$ we say that $X$ is a Hilbert $\cA$-module.

Let $X$ be a Hilbert $\cA$-module. We say that a map $T : X \rightarrow X$ is adjointable if there is a map $T^* :X \rightarrow X$ such that $\lip Tx,y \rip = \lip x, T^*y \rip$ for every $x,y \in X$. Every adjointable operator is automatically $\cA$-linear and continuous. We denote by $\cL(X)$ the C*-algebra of adjointable operators equipped with the operator norm. For $x,y\in X$ there is a special adjointable operator $\theta_{x,y} \in \cL(X)$ given by $\theta_{x,y}(z) = x \cdot \lip y,z \rip$. We will denote by $\cK(X) \lhd \cL(X)$ the closed ideal of generalized compact operators generated by $\theta_{x,y}$ with $x,y \in X$.

A $\cB$-$\cA$ $\ca$-correspondence is then just a (right) Hilbert $\cA$-module $X$ along with a $*$-homomorphism $\phi : \cB \rightarrow \cL(X)$ which is non-degenerate, i.e., $\overline{\Span} \{ \phi(B)X \}= X$ (this is sometimes called essential). If $X$ is an $\cA$-$\cA$ $\ca$-correspondence we will just call $X$ an $\cA$-correspondence. We think of $\phi$ as implementing a left action of $\cB$ on $X$, and we often write $b x$ for $\phi(b)x$. 

When $X$ is a $\cC$-$\cB$-correspondences and $Y$ a $\cB$-$\cA$-correspondence, we may form the interior tensor product $X \otimes_{\cB} Y$. Indeed, let $X\odot_{\cB} Y$ be the algebraic $\cB$-balanced tensor product. Then the formula
$$
\lip x \odot y, w \odot z \rip := \lip y, \lip x,w \rip \cdot z \rip,
$$
determines an $\cA$-valued sesquilinear form on $X\odot_{\cB} Y$, whose Hausdorff completion $X \otimes_{\cB} Y$ is a (right) Hilbert $\cA$-module. There is then a left $\cC$ action $\cC \rightarrow \cL(X \otimes_{\cB} Y)$ given by $c \cdot (x\odot y) = (c \cdot x) \odot y$ for $x \in X$, $y\in Y$ and $c\in \cC$. We will henceforth denote $X \otimes _{\B} Y$ by $X \otimes Y$ and simple tensors without reference to $\B$, where the algebra $\B$ is understood from context.

\subsection{Product systems over semigroups}

Let $\cA$ be a C*-algebra and $P$ a semigroup with identity $e$. A product system over $P$ with coefficients in $\cA$ is a collection of $\cA$-correspondences $X = (X_p)_{p\in P}$ such that
\begin{enumerate}
\item
$X_e$ is $\cA$ as an $\cA$-correspondence.
\item
For $p,q \in P$, there exists a unitary $\cA$-linear isomorphism $M_{p,q} : X_p \otimes X_q \rightarrow X_{pq}$
\item
The left and right multiplication on each $X_p$ are given via $M_{e,p}$ and $M_{p,e}$ for each $p\in P$ and we also have associativity in the sense that for $p,q,r \in P$,
$$
M_{p,qr}(I_{X_p} \otimes M_{q,r}) = M_{pq,r}(M_{p,q}\otimes I_{X_r})
$$ 
\end{enumerate}

We will denote $M_{p,q}(x\otimes y) = xy \in X_{pq}$ for every $x\in X_p$ and $y\in X_q$. We will also denote by $\phi_p : \cA \rightarrow \cL(X_p)$ the left action on $X_p$ for each $p\in P$. In particular, $\phi_{pq}(a)(xy) = (\phi_p(a)x)y$ for all $p,q\in P$, $a\in \A$ and $x\in X_p$, $y\in X_q$. 

Given $p \in P \setminus \{e\}$ and $q \in P$, the unitary $\cA$-linear map $M_{p,q} : X_p \otimes X_q \rightarrow X_{pq}$ induces a $*$-homomorphism $\iota_p^{pq} : \cL(X_p) \rightarrow \cL(X_{pq})$ via 
$$
\iota_{p}^{pq}(S) := M_{p,q} \circ (S \otimes id_{X_p})M_{p,q}^{-1}
$$
for each $S\in \cL(X_p)$. Alternatively, we have that the $*$-homomorphism $\iota_p^{pq}$ is given by the formula $\iota_p^{pq}(S)(xy) = (Sx)y$ for each $S\in \cL(X_p)$, $x\in X_p$ and $y\in X_q$. For $\iota_e^q$, we first define on $\cA \cong \cK(\cA)$ via $\iota_e^p(a) = \phi_p(a)$, and then extend uniquely to $\cL(\cA)$ via \cite[Proposition 2.5]{Lan95} to obtain a map $\iota_e^q : \cL(\cA) \rightarrow \cL(X_q)$. Finally, for notational purposes, we will define $\iota_p^r$ to be zero whenever $r \neq pq$ for all $q\in P$.

When $X=(X_p)_{p\in P}$ is a product system over a quasi-lattice ordered group $(G,P)$, we will say that $X$ is \emph{compactly aligned} if whenever $S\in \cK(X_p)$ and $T\in \cK(X_q)$ for some $p,q\in P$ with $p \vee q < \infty$, then $S \vee T := \iota_p^{p \vee q}(S) \iota_q^{p \vee q}(T) \in \cK(X_{p\vee q})$.

\subsection{Nica-Toeplitz representations} \label{intro;NC}
We next define representations of compactly aligned product systems over quasi-lattice ordered groups.

\begin{definition} \label{D:isom}
Suppose $(G,P)$ is a quasi-lattice ordered group, and $X$ is a compactly aligned product system over $P$ with coefficients in $\cA$. An \emph{isometric representation} of $X$ into a $\ca$-algebra $\B$ is a map $\psi: X \rightarrow \B$ comprised of linear maps $\psi_p : X_p \rightarrow \B$ for each $p\in P$ such that
\begin{enumerate}
\item
$\psi_e$ is a $*$-homomorphism from $\cA = X_e$ into $\B$.
\item
$\psi_p(x)\psi_q(y) = \psi_{pq}(xy)$ for all $p,q\in P$ and $x \in X_p$, $y\in X_q$.
\item
$\psi_p(x)^*\psi_p(y) = \psi_e(\lip x,y \rip)$ for all $p\in P$ and $x,y\in X_p$.
\end{enumerate}
\end{definition}

It is standard to show that each $\psi_p$ is contractive, and is isometric precisely when $\psi_e$ is injective. We will say that $\psi$ is non-degenerate provided that $\B \subseteq B(\H)$ and $\psi_e$ is non-degenerate. For each $p\in P$ we obtain a $*$-homomorphism $\psi^{(p)} : \cK(X_p) \rightarrow \B$ given by $\psi^{(p)}(\theta_{x,y}) = \psi_p(x)\psi_p(y)^*$ for all $x,y\in X_p$. 

If $X$ is the trivial product system over $(G,P)$, i.e., $X_p=\bbC$, for all $p \in P$, then an isometric representation of $X$ is simply a representation of $P$ as a semigroup of isometries.

We will say that an isometric representation $\psi$ of a compactly aligned product system $X$ is \emph{Nica-covariant} if for any $p,q\in P$ and $S\in \cK(X_p)$, $T\in \cK(X_q)$ we have that
\begin{equation} \label{D:NC1}
\psi^{(p)}(S)\psi^{(q)}(T) = 
\begin{cases}
\psi^{(p\vee q)}(S\vee T) & \text{if } p \vee q < \infty  \\
0 & \text{otherwise.}
\end{cases}
\end{equation}

In the case where $\B \subseteq B(\H)$ is a non-degenerate representation of $\B$, the above definition simplifies when we consider $\psi$ into $B(\cH)$. Indeed by \cite[Proposition 2.5]{Lan95} the $*$-homomorphism $\psi^{(p)}$ admits a unique strict-$\sot$ continuous extension on $\L(X_p)$, which will be still be denoted by $\psi^{(p)}$. Then $\psi: X \rightarrow B(\H)$ is Nica-covariant if and only if for any $p,q\in P$ we have 
\begin{equation} \label{D:NC2}
\psi^{(p)}(I)\psi^{(q)}(I) = 
\begin{cases}
\psi^{(p\vee q)}(I) & \text{if } p \vee q < \infty  \\
0 & \text{otherwise.}
\end{cases}
\end{equation}

This condition has the drawback that it applies only to representations into adjointable operators on Hilbert modules, but has the advantage that it works for arbitrary product systems which are not necessarily compactly aligned. With the advent of compactly aligned systems, condition (\ref{D:NC2}) was replaced by (\ref{D:NC1}). Both conditions are equivalent for concrete representations by \cite[Proposition 5.6]{Fow02}. Finally recall that in the case where $(G,P)$ is an abelian, lattice ordered group, we always have $p\vee q<\infty$ for any $p,q \in P$, so that the formulas simplify in both (\ref{D:NC1}) and (\ref{D:NC2}).

Each product system $X$ has a natural Nica-covariant isometric representation which we now describe. We denote by $\cF_X := \oplus_{p \in P}X_p$ the Fock space of $X$. We then define $l : X \rightarrow \cL(\cF_X)$ by $l_p(x)(y_q)_{q\in P} = (xy_q)_{q\in P}$ for each $p\in P$, $x\in X_p$, and $(y_q)_{q\in P} \in \cF_X$. We call $l$ the Fock representation, which is an isometric Nica-covariant representation of $X$ by \cite[Lemma 5.3]{Fow02}.

We denote by $\cN \cT_X$ the universal C*-algebra generated by a Nica-co\-variant representation of $X$, which exists due to \cite[Theorem 6.3]{Fow02}. Hence, there is an isometric Nica-covariant representation $i_X : X \rightarrow \cN \cT_X$ such that $\cN \cT_X$ is generated by the image of $i_{X}$ and for any other isometric Nica-covariant representation $\psi : X \rightarrow B(\cH)$ there exists a $*$-homomorphism $\psi_* : \cN \cT_X \rightarrow B(\cH)$ such that $\psi_* \circ i_{X,p} = \psi_p$, for every $p\in P$.

\begin{definition}
Let $\B$ be a C*-algebra and $G$ a discrete group. A (full) \emph{coaction} on $\B$ is an injective $*$-homomorphism $\delta : \B \rightarrow \B \otimes C^*(G)$ such that
\begin{enumerate}
\item $(\delta \otimes id) \delta = (id \otimes \delta_G)\delta$

\item $\delta(\B)(1 \otimes C^*(G))$ is dense inside $B \otimes C^*(G)$.
\end{enumerate}
where $\delta_G : C^*(G) \rightarrow C^*(G) \otimes C^*(G)$ is the $*$-homomorphism satisfying $\delta_G(i_G(g)) = i_G(g) \otimes i_G(g)$ for $g \in G$.
\end{definition}

By \cite[Proposition 3.5]{CLSV11} (see also \cite[Proposition 4.7]{Fow02}) there is a canonical gauge coaction $\delta_X : \cN \cT_X \rightarrow \cN \cT_X \otimes \ca(G)$ given by $\delta_X(i_{X,p}(x)) = i_{X, p}(x) \otimes i_G(p)$ for each $p\in P$ and $x\in X_p$. We will say that a representation $\psi : X \rightarrow B(\cH)$ is \emph{gauge-compatible} if there is a coaction $\delta$ of $G$ on $\ca(\{\psi_p(X_p)\}_{p\in P})$ such that $\delta(\psi_p(x)) = \psi_p(x) \otimes i_G(p)$ for all $x\in X_p$ and $p\in P$. When $\psi$ is gauge-compatible, the induced $*$-homomorphism $\psi_* : \cN \cT_X \rightarrow \ca(\{\psi_p(X_p)\}_{p\in P})$ will then be equivariant in the sense that $(\psi_* \otimes \id) \circ \delta_X = \delta \circ \psi_*$.

Since we are assuming in this paper that $G$ is abelian the concept of gauge compatibility reduces to a more familiar one.

\begin{proposition} \label{P:gauge}
Let $(G,P)$ be an abelian, lattice ordered group, and $X$ a compactly aligned product system over $P$ with coefficients in $\cA$. Then a representation $\psi : X \rightarrow B(\cH)$ is gauge-compatible if and only if there exists a gauge action 
$$\hat{\alpha}_X \colon \hat{G} \longrightarrow \Aut\big(\ca(\{\psi_p(X_p)\}_{p\in P})\big)$$ such that $\hat{\alpha}_{X, \gamma}(\psi_p(x)) =\gamma(p) \psi_p(x) $ for all $x\in X_p$, $p\in P$ and $\gamma \in \hat{G}$, the Pontryagin dual of $G$. 
\end{proposition}

\begin{proof}
For notational simplicity, let $\B := \ca(\{\psi_p(X_p)\}_{p\in P})$. Assume first that $\psi$ is gauge-compatible via a coaction $\delta$ of $G$. Let $$F \colon \ca(G) \longrightarrow \mathrm{C}(\hat{G})$$ be the Fourier transform. Then for each $\gamma \in \hat{G}$, the desired $\hat{\alpha}_{X, \gamma}\colon \B \rightarrow \B$ is given by the composition 
\[\xymatrix{
\B \ar[r]^{\delta\phantom{000000}}  & \B \otimes \ca(G)\ar[r]^{\id\otimes F} &  \B \otimes \mathrm{C}(\hat{G})\ar[r]^{\phantom{0000}\id\otimes e_{\gamma}}& \B,}
\]
where $e_{\gamma}\colon \mathrm{C}(\hat{G})\rightarrow \bbC$ denotes the evaluation at $\gamma \in \hat{G}$.

Conversely, assume that $\psi$ admits a gauge action $\hat{\alpha}_X \colon \hat{G} \longrightarrow \Aut(\B)$ with $\hat{\alpha}_{X, \gamma}(\psi_p(x)) =\gamma(p) \psi_p(x) $ for all $x\in X_p$, $p\in P$ and $\gamma \in \hat{G}$. Then the desired coaction $\delta$ comes from the composition 
\[\xymatrix{
\B \ar[r]^{R\phantom{000000}}  & \B \otimes \mathrm{C}(\hat{G}) \ar[r]^{\id\otimes F^{-1}} &  \B \otimes \ca(G),}
\]
where $R \colon \B \rightarrow  \B \otimes \mathrm{C}(\hat{G}) $ is defined by $R(b) (\gamma) := \hat{\alpha}_{X,\gamma}(b)$ for $b \in \B$, $\gamma \in \hat{G}$ and we identify $\B \otimes \mathrm{C}(\hat{G})$ with $\mathrm{C}(\hat{G},\B)$.
\end{proof}

In \cite[Theorem 7.2]{Fow02} Fowler provides a condition on an isometric Nica-covariant representation $\psi$ to induce a faithful $*$-representation $\psi_*$ of $\cN \cT_X$. As we saw earlier, for an isometric Nica-covariant representation $\psi$ we have the projections $\{\psi^{(p)}(I)\}_{p\in P}$. Each $\psi^{(p)}(I)$ is the SOT limit of (any) contractive approximate identity for $\psi^{(p)}(\cK(X_p))$. Furthermore, since each $X_p$ is essential, $\psi^{(p)}(I)$ projects onto $\psi_p(X_p)\cH$. Hence, we get the following variant of Fowler's result, which we state for abelian, lattice ordered groups.

\begin{theorem}[Fowler] \label{T:toep-uniq}
Suppose $(G,P)$ is an abelian, lattice ordered group, $X$ is a compactly aligned product system over $P$ with coefficients in $\cA$, and $\psi :X \rightarrow B(\cH)$ an isometric Nica-covariant representation. Suppose that for any finite set $F\subseteq P \setminus \{e\}$ the representation $A \rightarrow B(\cH)$ given by
$$
a \mapsto \psi_e(a) \prod_{p\in F} \big(I_{\cH} -\psi^{(p)}(I)\big)
$$
is faithful. Then the induced $*$-representation $\psi_* : \cN \cT_X\rightarrow B(\H)$ is faithful.
\end{theorem}

As an easy application of this theorem, one can show that the $*$-homo\-morphism $l_* : \cN \cT_X \rightarrow \cL(\cF_X)$ induced from the Fock representation $l$ is faithful.

\section{The C*-envelope of the Nica tensor algebra} \label{Sec:envelope}

In this section we characterize the C*-envelope of the Nica tensor algebra $\cN \cT^+_X$ of a product system $X$ over a (discrete) abelian, lattice ordered group $(G,P)$. The Nica tensor algebra of the product system $X$ is the norm closed algebra
$$
\cN \cT^+_X := \overline{\alg} \{ (i_{X,p})(X_p) \mid p \in P\}.
$$
By the remark following Theorem~\ref{T:toep-uniq}, $\N\T^+_X$ is completely isometrically isomorphic to the non-selfad\-joint operator algebra generated by the image of the Fock representation $l \colon X \rightarrow \L(\F_X)$. It is naturally a subalgebra of $\cN \cT_X$, and is also the universal norm-closed operator algebra generated by a Nica-covariant isometric representation of $X$.

In \cite{CLSV11} the authors investigate the existence of a pair $(\N\O_X^r, j)$ which is co-universal for isometric, Nica-covariant, gauge-compatible representations of $X$ in the following sense: $\N\O_X^r$ is a $\ca$-algebra and $j : X \rightarrow \N\O_X^r$ is a Nica-covariant representation satisfying the following properties
\begin{enumerate}
\item 
$j_e$ is faithful,
\item
$j_*$ is a gauge-compatible surjection, where $j_* : \cN \cT_X \rightarrow \N\O_X^r$ is the canonical $*$-homomorphism induced by $j$, and
\item for any gauge-compatible Nica-covariant isometric representation $\psi : X \rightarrow \B$ for which $\psi_e$ is faithful, there is a surjective $*$-homomor- phism $q : \ca (\{\psi_p(X_p)\}_{p \in P}) \rightarrow \N\O_X^r$ such that $q \circ \psi_p(\xi) = j_p(\xi)$, for all $\xi \in X_p$ and $p \in P$.
\end{enumerate}

It goes without saying that the existence of such a pair $(\N\O_X^r, j)$ is not immediate; nevertheless if it exists then it is unique \cite[Theorem 4.1]{CLSV11}. Most of \cite{CLSV11} is devoted to showing that $(\N\O_X^r, j)$ exists in many cases. However many cases are also left open in \cite{CLSV11}; in particular, the case for product systems over abelian, lattice ordered groups is not completely resolved. Following a suggestion by Hangfeng Li, the authors of \cite{CLSV11} also raise the possibility that the existence of $(\N\O_X^r, j)$ is somehow related to Arveson's concept of a $\ca$-envelope. However this point is also not clarified in \cite{CLSV11}. In what follows the pair $(\N\O_X^r, j)$ will be simply denoted as $\N\O_X^r$.

An algebra closely related to $\N\O_X^r$ is the $\ca$-algebra $\A \times_X P$ of Sehnem \cite{Seh+}. Specialized to compactly aligned product systems, the algebra $\A \times_X P$ is determined by the following properties: there exists a Nica-covariant representation $j': X \rightarrow \A \times_X P$ that induces a canonical $*$-homomorphism $j'_* : \N \T_X \rightarrow \A \times_X P$ so that
\begin{itemize}
\item[(A)] $j'_*$ restricts to an injection on $X_e = \A$, and
\item[(B)] any representation of $\A \times_X P$ which is faithful on $j'_*(\A)$ is also faithful on the fixed point algebra $(\A \times_X P)^{\delta}$, where $\delta$ is the canonical gauge coaction of $G$ on $\A \times_X P$.
\end{itemize}

The existence of such an algebra was originally hypothesized by Sims and Yeend in \cite{SY11}. The motivation came from the various gauge invariant uniqueness theorems, which are an important tool for recolonizing such algebras uniquely in concrete situations. Subsequently, the existence of such an algebra was established in many important cases, including $(G,P)=(\bbZ^d,\bbN^d)$, by Carlsen, Larsen, Sims and Vittadello in \cite{CLSV11}. In complete generality, the existence of $\A \times_X P$ was verified recently by Sehnem \cite{Seh+}. Sehnem defines the algebra $\A \times_X P$ as a universal object for a particular class of representations of $X$ that she calls strongly covariant and then establishes properties (A) and (B) in \cite[Theorem 3.10]{Seh+}. This approach has the additional advantage that $\A \times_X P$ admits a more concrete presentation than that of $\N\O_X^{\,r}$.

Let $(G,P)$ be an abelian, lattice ordered group. If $\psi : X \rightarrow B(\cH)$ and $V: P \rightarrow B(\cK)$ are isometric representations \footnote{As we pointed out in Subsection~\ref{intro;NC}, by $V: P \rightarrow B(\cK)$ an isometric representation we simply mean a representation of $P$ by isometries on $\K$.}, then we define a representation 
\begin{equation} \label{eq:tensor rep}
\psi \otimes V : X \longrightarrow B(\cH \otimes \cK); \, X_p\ni x \longmapsto \psi_p(x) \otimes V_p, \quad p \in P.
\end{equation}
Clearly $\psi \otimes V$ is an isometric representation of $X$.

\begin{lemma} \label{L:Nicatensor}
Let $(G,P)$ be an abelian, lattice ordered group and let $X$ be a compactly aligned product system over $P$ with coefficients in $\cA$. Assume that $\psi : X \rightarrow B(\cH)$ and $V: P \rightarrow B(\cK)$ are isometric representations. 
\begin{enumerate}
\item[(a)]
If both $\psi$ and $V$ are Nica-covariant then $\psi \otimes V$ is Nica-covariant.
\item[(b)]
If $\psi$ is a Nica-covariant representation and $U$ is a unitary representation of $P$, then $\psi \otimes U$ is Nica-covariant.
\end{enumerate}
\end{lemma} 

\begin{proof}  
As for $(a)$, according to (\ref{D:NC2}), and as any two elements in $P$ have a least upper bound, we need to verify that
$$
(\psi\otimes V)^{(p)}(I)(\psi \otimes V)^{(q)}(I) = 
(\psi \otimes V)^{(p\vee q)}(I)
$$
Now notice that for any $\xi,\eta \in X_p$ we have that
\begin{align*}
(\psi\otimes V)^{(p)}(\theta_{\xi,\eta}) = \psi_p(\xi)\psi_p(\eta)^* \otimes V_pV_p^* = \psi^{(p)}(\theta_{\xi,\eta}) \otimes V_pV_p^*.
\end{align*}
By taking linear combinations and approximating in the strict topology, we get $(\psi\otimes V)^{(p)}(I) = \psi^{(p)}(I)\otimes V_pV_p^*$, and similarly $(\psi\otimes V)^{(q)}(I) = \psi^{(q)}(I)\otimes V_qV_q^*$. Therefore we have
\begin{align*}
(\psi\otimes V)^{(p)}(I)(\psi \otimes V)^{(q)}(I) &= \psi^{(p)}(I)\psi^{(q)}(I)\otimes V_pV_p^*V_qV_q^* \\
				&=\psi^{(p \vee q)}(I) \otimes V_{p\vee q}V^*_{p \vee q} \\
				&= (\psi \otimes V)^{(p\vee q)}(I) 
\end{align*}
as desired. For a proof of $(b)$, since $(\psi \otimes U)^{(p)} = \psi^{(p)} \otimes I_{\cK}$, the conclusion follows as before. 
\end{proof}

\begin{lemma} \label{lemma;NOr}
Let $(G,P)$ be an abelian, lattice ordered group and let $X$ be a compactly aligned product system over $P$. Suppose $\psi: X \rightarrow \B$ is an isometric Nica-covariant representation of $X$ into a $\ca$-algebra $\B$ such that 
\begin{enumerate}
\item 
$\psi_e$ is faithful, and
\item
$\psi$ is gauge-compatible.
\end{enumerate}
Then $\psi_*$ is completely isometric on $\cN \cT_X^+$. 
\end{lemma}

\begin{proof}
By representing $\B$ appropriately on a Hilbert space $\H$ we may assume that $\psi \colon X \rightarrow B(\H)$ is non-degenerate. Since $\psi$ is gauge-compatible, there is a $*$-homomorphism
\[\delta : \ca(\{\psi(X_p)\}_{p\in P}) \rightarrow \ca(\{\psi(X_p)\}_{p\in P}) \otimes \ca(G)
\]
 such that $\delta(\psi_p(x)) = \psi_p(x) \otimes U_p$, $p \in P$, where $g \mapsto U_g$ is the left regular representation of $G$ on $\ell^2(G)$. Hence, we have the commutative diagram
\[
\xymatrix{ \N\T_X \ar[d]^{\psi_*} \ar[rr]^{\delta_X \phantom{cccccccc}} & &  \N\T_X \otimes \ca(G) \ar[d]^{\psi_*\otimes\id}\\
\ca(\{\psi(X_p)\}_{p\in P})  \ar[rr]_{\delta\phantom{ccccccccc}}  && \ca(\{\psi(X_p)\}_{p\in P}) \otimes  \ca(G))}
\]
In order to prove that  $\psi_*$ is completely isometric on $\cN \cT_X^+$, it is therefore sufficient to show that $(\psi_* \otimes \id) \circ \delta_X$ is completely isometric on $\cN \cT_X^+$. 

Consider the isometric representation 
\[
\psi\otimes V \colon X \longrightarrow B\big(\H \otimes \ell^2(P)\big); X_p \ni x \longmapsto \psi_p(x)\otimes V_p, \quad p \in P,
\]
where $p \mapsto V_p$ is the left regular representation of $P$. By Lemma~\ref{L:Nicatensor} we have that $\psi\otimes V$ is Nica-covariant. We claim that $\psi\otimes V$ satisfies the conditions of Fowler's Theorem~\ref{T:toep-uniq}. Indeed, let $P_e $ be the projection onto scalar multiples of the characteristic function of $e \in P$ in $\ell^2(P)$. Then for any $ p \in P\backslash \{e\}$ we have
\[
(\psi\otimes V)^{(p)}(I)(I\otimes P_e)= (I\otimes P_e)(\psi\otimes V)^{(p)}(I) = 0,
\]
so for all $p\in P \setminus \{e\}$ we have
 $\big(I - (\psi \otimes V)^{(p)}(I)\big)(I\otimes P_e)= I\otimes P_e$. 
Hence, for each finite set $F\subseteq P \setminus \{e\}$, the map 
$$
X_e \longrightarrow B(\cH \otimes \ell^2(P)); \, a \longmapsto (\psi \otimes V)_e(a) \prod_{p\in F}\big(I - (\psi \otimes V)^{(p)}(I)\big)
$$
is injective when restricted to the reducing subspace $(I \otimes P_e)(\cH \otimes \ell^2(P))$ as it reduces to the map
\[
X_e \longrightarrow B(\cH \otimes \ell^2(P)); \, a \longmapsto \psi(a)\otimes P_e.
\]
Therefore, Fowler's Theorem~\ref{T:toep-uniq} implies that the induced $*$-homomorphism $$(\psi\otimes V)_* : \cN \cT_X \longrightarrow B(\cH \otimes \ell^2(P))$$ is injective. In particular, it is completely isometric on $\cN \cT^+_X$.

Recall that by Lemma \ref{L:Nicatensor} we have that $\psi \otimes U$ is a Nica-covariant isometric representation of $X$ and so $(\psi \otimes U)_*$ makes sense as a representation of $\N \T_X$. Now let $Q$ be the projection from $\ell^2(G)$ to $\ell^2(P)$. Then 
\[
(I\otimes Q)(\psi\otimes U)(x_p)(I\otimes Q)=(\psi\otimes V)(x_p), 
\]
for all $x_p \in X_p$, $p \in P$. Since the range of $Q$ is invariant for polynomials in $U_p$ for $p\in P$, the above equation promotes to
\begin{equation}\label{eq;dil}
(I\otimes Q)(\psi\otimes U)_*(x)(I\otimes Q)=(\psi\otimes V)_*(x), 
\end{equation}
for all $x \in \N\T_X^+$. As we saw earlier, $(\psi\otimes V)_*$ is completely isometric on $\cN \cT^+_X$. Therefore (\ref{eq;dil}) implies the same for $(\psi\otimes U)_*$. However
\[
(\psi\otimes U )_*=(\psi_*\otimes\id)\circ \delta_X
\]
and so $(\psi_*\otimes\id)\circ \delta_X$ is completely isometric on $\cN \cT^+_X$, as desired.
\end{proof}

Before stating our main result, we need to make an observation. If $\C$ is a $\ca$-cover for an operator algebra $\A$ and $\J\subseteq \C$ is the Shilov ideal of $\A\subseteq \C$, then $\J$ is left invariant by any automorphism of $\C$ that leaves $\A$ invariant. This is due to the fact that $\J$ is the \textit{largest} boundary ideal in $\C$.

\begin{theorem} \label{T:CNP-envelope}
Let $(G,P)$ be an abelian, lattice ordered group. Let $X$ be a compactly aligned product system over $P$ with coefficients in $\A$. Then the $\ca$-algebras
\begin{itemize}
\item[\textup{(i)}] $\ca_e(\N\T_X^+)$, the $\ca$-envelope of the Nica tensor algebra $\N\T^+_X$,
\item[\textup{(ii)}] $ \cN \cO_X^r$, the co-universal $\ca$-algebra for gauge-compatible, Nica-cov\-ariant representations of X, of Carlsen, Larsen, Sims and Vittadello \cite{CLSV11}, and
\item[\textup{(iii)}] Sehnem's $\ca$-algebra $\A \times_XP$ \cite{Seh+},
\end{itemize}are mutually isomorphic via maps that send generators to generators.
\end{theorem}

\begin{proof} First we show that $\ca_e(\N\T_X^+) \simeq \cN \cO_X^r$ canonically. To do this, it will suffice to show that $\ca_e(\N\T_X^+) $ satisfies the properties that determine $\cN \cO_X^r$. 

Let $i_X\colon X \rightarrow \N\T_X$ be the universal Nica-covariant isometric representation of $X$ and let $\J \subseteq \N\T_X$ be the Shilov ideal of $\N\T_X^+ \subseteq \N\T_X$. Now  \cite[Proposition 5.9]{Fow02} shows that if we compose $i_X$ with any representation of $\N\T_X$, we obtain a Nica-covariant isometric representation of $X$. Therefore by taking quotient with $\J$, we obtain a Nica-covariant isometric representation
\[
j : X \longrightarrow \N\T_X\slash \J=\cenv(\N\T_X^+),
\]
which promotes to a representation $j_*: \N\T_X \rightarrow \cenv(\N\T_X^+)$.
By the discussion preceding the theorem, the canonical gauge action of $\hat{G}$ on $\N\T_X$ leaves invariant the Shilov ideal $\J$. By Proposition~\ref{P:gauge}, $\cenv(\N\T_X^+)$ admits a canonical gauge coaction and $j_*$ becomes gauge-compatible. In addition, $j_e$ is faithful because $\A \subseteq \N\T_X^+$. Therefore properties (i) and (ii) of the definition for $\N\O_X^r$ hold for the pair $(\cenv(\N\T^+_X), j)$.

Now let $\psi : X \rightarrow \B$ be a gauge-compatible, Nica-covariant isometric representation with $\psi_e$ faithful. By Lemma~\ref{lemma;NOr}, the restriction of $\psi_*$ on $\N\T_X^+$ is completely isometric and so by the defining properties of the $\ca$-envelope, there exists a surjective $*$-homomorphism $q:\ca(\psi(X))\rightarrow \cenv(\N\T_X^+)$ so that $j_*=q\circ\psi_*$. This establishes property (iii) and so the pair $(\cenv(\N\T^+_X), j)$ satisfies all the defining properties of $\cN \cO_X^r$, as desired.

Having established the existence of $\N\O_X^r$, the canonical isomorphism $\N\O_X^r \simeq \A\times_X P$ follows now from $(A)$ and $(B)$. Indeed, consider Sehnem's universal representation $j' \colon X \rightarrow \A \times_X P$, which is faithful on $\A$ and gauge-compatible by \cite[Theorem 3.10.(C3)]{Seh+} and Nica-covariant by \cite[Proposition 4.2]{Seh+}. By property (iii) of $\N\O_X^r$ (co-universality), there exists a $*$-homomorphism $\pi \colon \A \times_X P \rightarrow \N\O_X^r$ so that $j_* = \pi \circ j'_*$. By \cite[Theorem 3.10.(C3)]{Seh+} the $*$-homorphism $\pi$ is faithful on the fixed point algebra $(\A \times_X P)^{\delta}$ of the canonical coaction $\delta$ of $G$ on $\A\times_X P$. Thus, as the conditional expectation onto $(\A \times_X P)^{\delta}$ is faithful, and as $\pi$ is equivariant, it is easy to see that $\pi$ must be injective on all of $\A \times_X P$.
\end{proof}

One of the significant features of Theorem~\ref{T:CNP-envelope} is that it establishes the existence of $\N\O_X^r$ for any compactly aligned product system over an \textit{abelian}, lattice ordered group, thus resolving the investigation initiated in \cite{CLSV11} for this important case. Stated formally

\begin{corollary} \label{C:CNP-envelope}
Let $(G,P)$ be an abelian, lattice ordered group. If $X$ is any compactly aligned product system over $P$ with coefficients in $\A$, then the co-universal $\ca$-algebra $ \cN \cO_X^r$ for gauge-compatible, Nica-covariant representations of X, of Carlsen, Larsen, Sims and Vittadello \cite{CLSV11} exists.
\end{corollary}

It is also important to note that Corollary~\ref{C:CNP-envelope} is a purely selfadjoint result which is obtained through the use of non-selfadjoint operator algebra theory.

The actual definition of Sehnem's algebra $\A\times_X P$ \cite{Seh+} is quite involved. Nevertheless, in the case where every bounded subset of the semigroup $P$ has a maximal element, e.g., when $(G,P) =(\bbZ^d, \bbN^d)$, the definition of $\A\times_XP$ simplifies considerably and is worth repeating here.

Suppose $(G,P)$ is abelian and lattice ordered and $X$ is a compactly aligned product system over $P$ with coefficients in $\cA$. We denote $I_e(X)= \cA$, and for each $q\in P \setminus \{e\}$ write $I_q(X):= \cap_{e<p\leq q} \Ker (\phi_p)$. We then denote
$$
\tilde{X}_q:= \oplus_{p\leq q} X_p \cdot I_{p^{-1}q}(X).
$$
The homomorphism implementing the left action on the $\cA$-correspondence $\tilde{X_q}$ is denoted by $\tilde{\phi}_q$, and we say that $X$ is $\tilde{\phi}$-injective if $\tilde{\phi}_q$ is injective for all $q\in P$. By \cite[Lemma 3.15]{SY11}, if every bounded subset of $P$ has a maximal element, every product system over $P$ is automatically $\tilde{\phi}$-injective. 

Recalling the definitions of the maps $\iota_p^q$, we obtain homomorphisms $\tilde{\iota}_p^q : \cL(X_p) \rightarrow \cL(\tilde{X}_q)$ for all $p,q\in P$ with $p\neq e$ defined by 
$$
\tilde{\iota}_p^q(T) =\Big( \oplus_{p \leq r \leq q}\iota_p^r(T)\Big) \oplus \left(\oplus_{p \nleq r \leq q} 0_{X_r\cdot I_{r^{-1}q}(X)}\right).
$$ 
When $p=e$, similarly to the above we obtain a homomorphism $\tilde{\iota}_e^q : \cK(X_e) \rightarrow \cL(\tilde{X}_q)$.

We say that a Nica-covariant isometric representation $\psi$ of $X$ is \emph{Cuntz-Nica-Pimsner} covariant (a CNP representation) if given any finite set $F\subseteq P$ and compact operators $T_p \in \K(X_p)$ with $p\in F$, satisfying 
\begin{align} \label{den; eqCNP}
\sum_{p\in F} \tilde{\iota}_p^q(T_p) = 0
\end{align}
for ``sufficiently large" $q\in P$, we have $\sum_{p\in F}\psi^{(p)}(T_p) = 0$. Here ``sufficiently large" $q\in P$ means that given any $s\in P$, there exists $r\geq s$ so that \eqref{den; eqCNP} is satisfied for all $q\geq r$.

We denote by $\N \O_X$ the universal C*-algebra generated by a (universal) CNP covariant representation $j_X : X \rightarrow B(\H)$. One of the main results \cite[Theorem 4.1]{SY11} in the paper of Sims and Yeend is that in the case where $X$ is $\tilde{\phi}$-injective, the representation $j_{X,p}:=(j_X)_p$ is isometric on each $X_p$. Then, \cite[Proposition 4.6]{Seh+} shows now that $\A\times_X P$ and $\N\O_X$ are canonically isomorphic. We will use this fact in Section~\ref{S:SZ}.

For the moment recall that a product system $X$ over an abelian, lattice ordered semigroup $(G,P)$ is said to be \emph{injective} if the left action $\phi_p$ of each $X_p$ is injective, and \emph{regular} if additionally the image of $\phi_p$ is in $\cK(X_p)$ for all $p \in P$. By \cite[Corollary 5.2]{SY11} we know that for regular product systems $X$ with $(G,P)$ quasi-lattice ordered, an isometric representation $\psi$ of $X$ is CNP if and only if $\psi^{(p)} \circ \phi_p = \psi_e$ for all $p \in P$. Hence, we also obtain

\begin{corollary}
Let $X$ be a \emph{regular} product system over an abelian, lattice ordered semigroup $(G,P)$. Then $\ca_e(\N\T^+_X)$ is the universal C*-algebra generated by the image of representations $\psi_*$ where $\psi$ is an isometric representation of $X$ such that $\psi^{(p)} \circ \phi_p = \psi_e$ for all $p \in P$.
\end{corollary}

\begin{remark}
In a private communication, and by using completely different methods, Evgenios Kakariadis characterized the C*-envelope of injective product systems $X$ over $(\bbZ^d, \bbN^d)$, when $X$ satisfies the additional condition of strong compact alignment (see \cite[Definition 2.2]{DK+}).
\end{remark}

\section{Dilations of Nica-Toeplitz representations} \label{S:SZ}

In this section we answer a question of Skalski and Zacharias, which appears at the end of \cite{SZ08}. For that purpose, we explain some further necessary background on non-commutative boundary theory.

Suppose $\cA$ is an operator algebra generating a C*-algebra $\cB$. Recall that we say that $\rho : \cA \to B(\cH)$ is a \emph{representation} if $\rho$ is a completely contractive homomorphism. When $\A$ is unital, we say that a unital representation $\rho:\cA \to B(\cH)$ has the {\em unique extension property} if the only unital completely positive extension to $\cB$ is a $*$-representation.

For a unital operator algebra $\cA$ and a unital representation $\varphi: \cA \rightarrow B(\cH)$, a unital representation $\psi : \cA \rightarrow B(\cK)$ is said to {\em dilate} $\varphi$ if there is an isometry $V: \cH \rightarrow \cK$ such that for all $a\in \cA$ we have $\varphi(a) = V^*\psi(a)V$. Since $V$ is an isometry, we can identify $\cH \cong V(\cH)$ as a subspace of $\cK$, so that $\psi$ dilates $\varphi$ if and only if there is a larger Hilbert space $\cK$ containing $\cH$ such that for all $a \in \cA$ we have that $\varphi(a) = P_{\cH} \psi(a)|_{\cH}$ where $P_{\cH}$ is the projection onto $\cH$. We say that a unital representation $\rho : \cA \rightarrow B(\cK)$ is {\em maximal} if whenever $\pi$ is a unital representation dilating $\rho$, then in fact $\pi = \rho \oplus \psi$ for some unital representation $\psi$.

Based on ideas of Agler \cite{Agl88}, Muhly and Solel from \cite{MS98a}, Dritschel and McCullough \cite[Theorem 1.1]{DM05} (see also \cite{Arv+}) showed that a unital representation $\rho : \cA \rightarrow B(\cK)$ is maximal if and only if it has the unique extension property. Dritschel and McCullough \cite[Theorem 1.2]{DM05} then used this to show that every unital representation $\rho$ on $\cA$ can be dilated to a \emph{maximal} unital representation $\pi$ on $\cA$.

Suppose now that $X$ is a compactly aligned product system with coefficients in $\cA$ over a quasi-lattice ordered semigroup $(G,P)$. We will say that a Nica-covariant isometric representation $\psi : X \rightarrow B(\cH)$ of $X$ is \emph{fully-coisometric} if $\psi_p(X_p)\cH = \psi_e(\cA)\cH$ for every $p\in P$.

When $(G,P) = (\bZ^d, \bN^d)$, we have a standard set of generators $\{e_i\}_{i=1}^d$ for $\bN^d$, so that an isometric representation $\psi$ can be parameterised as a $(d+1)$-tuple $(\sigma^{\psi}, T_1^{\psi},...,T_d^{\psi})$ where $\sigma^{\psi}:= \psi_e$ is a representation of the coefficient algebra $\cA$, and $T_i^{\psi}:=\psi_{e_i}$ is an isometric representation of $X_{e_i}$ in the sense of \cite[Definition 1.1]{SZ08}. In this case, the $d+1$-tuple $(\sigma^{\psi}, T_1^{\psi},...,T_d^{\psi})$ becomes an isometric representation of $X$ in the sense of \cite[Definition 1.2]{SZ08}. Nica covariance of $\psi$ is equivalent to the assumption that $(T_1^{\psi},...,T_d^{\psi})$ are doubly commuting in the sense of \cite[Definition 2.1]{SZ08}, and when $\psi$ is non-degenerate, it is fully-coisometric if and only if $(\sigma^{\psi}, T_1^{\psi},...,T_d^{\psi})$ is fully-coisometric as in \cite[Definition 1.1]{SZ08}. In fact, the above forms a bijection between Nica-covariant isometric representations $\psi$ of $X$, and doubly-commuting isometric representations $(\sigma,T_1,...,T_d)$ as in definition \cite[Definition 1.1]{SZ08}. We also have that when an isometric representation $\psi$ is fully-coisometric, it is Cuntz-Pimsner covariant in the sense that $\psi^{(p)} \circ \phi_p = \psi_e$(see \cite[Section 1]{Fow02}), so that by \cite[Proposition 5.4]{Fow02} it is automatically Nica-covariant.

In the classical context of a unitary dilation of isometries, a theorem of It\^o shows that every commuting $d$-tuple of isometries dilate to a commuting $d$-tuple of unitaries. In the hope of generalizing It\^o's theorem, Skalski and Zacharias (see \cite[Section 5]{SZ08}) ask when an isometric representation $(\sigma,T_1,...,T_d)$ of a product system $X$ over $\bN^d$ has an isometric and fully-coisometric dilation $(\pi,U_1,...,U_d)$. Laca's theorem \cite{Lac00} generalizes It\^o's theorem to show that every isometric representation $V: P \rightarrow B(\cH)$ of an Ore semigroup $(G,P)$ dilates to a unitary representation. This includes all abelian, lattice ordered groups $(G,P)$.

When $X$ is regular, due to \cite[Corollary 5.2]{SY11}, and as the projection $\psi^{(p)}(I)$ onto $\psi_p(X_p)\cH$ is the SOT limit of (any) contractive approximate identity for $\psi^{(p)}(\cK(X_p))$, we have that $\psi$ is CNP if and only if each $\psi_p$ is fully-coisometric. The following then yields an answer to the question of Skalski and Zacharias when $X$ is regular and $\psi$ is Nica-covariant, and is a generalization of \cite[Theorem 5.4]{SZ08} to the abelian, lattice ordered semigroup context.

\begin{theorem} \label{T:unitary-dil}
Let $X$ be a regular product system over an abelian, lattice ordered semigroup $(G,P)$ and let $\psi$ be a Nica-covariant isometric representation of $X$ on a Hilbert space $\cH$. Then there is an isometric and fully-coisometric representation $\tilde{\psi}$ of $X$ on $\cK$ that dilates $\psi$ in the sense that $\cH \subseteq \cK$ and $\psi_p(x_p)=P_{\cH}\tilde{\psi}_p(x_p)|_{\cH}$ for every $x_p \in X_p$.
\end{theorem}

\begin{proof}
Assume that $\N\T_X^+$ is not unital. Since $\N\T_X^+$ has a contractive approximate unit, by \cite[Remark 2.1.8]{BLM04} no $\ca$-cover of $\N\T_X^+$ is unital.

Let $\psi$ be a Nica-covariant isometric representation of $X$. This induces a representation $\psi_*$ of $\N\T_X$. Thus, a representation $\psi_* : \cN \cT_X^+ \rightarrow B(\cH)$ (which we denote the same way) is obtained by restriction. By Meyer's Theorem \cite[Section 3]{Mey01}, we have a unital representation 
\[
\psi_*^1 : (\cN \cT_X^+)^1 \longrightarrow B(\cH).
\]
By the Dritschel-McCullough theorem \cite[Theorem 1.2]{DM05}, we have that $\psi_*^1$ dilates to a \emph{maximal} unital representation 
\[
\tilde{\psi_*^1} : (\cN \cT_X^+)^1 \longrightarrow B(\cK),
\]
which by \cite[Theorem 1.1]{DM05} has the unique extension property. By Theorem \ref{T:CNP-envelope} and Meyer's Theorem, $(\N\O^r_X)^1$ becomes a $\ca$-cover of $(\cN \cT_X^+)^1$, via a map that sends generators to generators canonically. We may therefore extend $\tilde{\psi_*^1}$ to a $*$-representation of $(\N\O^r_X)^1$, which we denote with the same symbol.
 We now define 
 \[
 \tilde{\psi} : X \longrightarrow B(\cK); \ \xi_p \mapsto \tilde{\psi_*^1}(i_{X, p}(\xi_p)), \ p \in P, \ \xi_p \in X_p,
 \]
 to obtain an isometric Nica-covariant representation $\tilde{\psi} : X \rightarrow B(\cK)$ that dilates $\psi$. Since 
 \[
 j_X\colon X \longrightarrow  \N\O^r_X\simeq \A \times_X P
 \]
 is by definition a CNP representation, we have that $ \tilde{\psi}  = \tilde{\psi_*^1} \circ j_X$ is also a CNP representation. By the remarks preceding the theorem, $\tilde{\psi}$ is an isometric and fully-coisometric representation that dilates $\psi$. This completes the proof in the non-unital case.
 
 If $\N\T_X^+$ is unital, then a repetition of the above arguments, without resorting to unitizations, yields a proof.
\end{proof}

The following is the positive answer to the original question of Skalski and Zacharias, under the assumption of Nica-covariance. In Example \ref{E:Nica-needed} we will see that this is optimal.

\begin{corollary} \label{C:unit-dil-SZ}
Let $X$ be a \emph{regular} product system over $\bN^d$, and suppose that $(\sigma,T_1,...,T_d)$ is a \emph{doubly-commuting} isometric representation of $X$ on a Hilbert space $\cH$. Then there is an isometric and \emph{fully-coisometric} representation $(\pi,U_1,...,U_d)$ of $X$ that dilates $\psi$.
\end{corollary}

\section{Examples and comparisons} \label{Sec:examples}

In this section we discuss our results in the case of operator algebras arising from higher rank graphs and show by example that Theorem \ref{T:unitary-dil} is optimal. Our main result from Section \ref{Sec:envelope} has consequences on operator algebras arising from topological higher rank graphs in the sense of Yeend \cite{Yee06,Yee07} or those arising from dynamical systems as in \cite{DFK14}, and this is discussed at the end of the section.

We first collect some terminology on higher rank graphs and their product systems. More details can be found in \cite{RS03}, \cite{RSY03} and \cite{RSY04}. Let $G= (V,E,r,s)$ be a directed graph, and partition $E=E_1 \cup ... \cup E_d$ such that each edge carries a unique color from a selection of $d$ colors. Denote by $E^{\bullet}$ the set of all paths in $G$. We may define a multi-degree function $d: E^{\bullet} \rightarrow \bN^d$ by $d(\lambda) = (n_1,...,n_d)$, where $n_i$ is the number of edges in $\lambda$ from $E_i$.

A higher rank structure on $G$ is an equivalence relation $\sim$ on $E^{\bullet}$ such that for all $\lambda \in E^{\bullet}$ and $p,q \in \bN^d$ with $d(\lambda) = p + q$, there exist \emph{unique} $\mu,\nu \in E^{\bullet}$ with $s(\lambda) = s(\nu)$, $r(\lambda) = r(\mu)$, such that $d(\mu)=p$ and $d(\nu) = q$ and $\lambda \sim \mu \nu$. We denote $\Lambda = E^{\bullet} / \sim$ and keep denoting $d: \Lambda \rightarrow \bN^d$ the multi-degree map. We call the pair $(\Lambda,d)$ a higher rank graph, so that this way, it becomes a higher rank graph as in \cite[Definition 2.1]{RSY03}. We will keep denoting elements in $\Lambda$ by $\lambda$ with the understanding that they may be represented in various ways. For each $p\in \bN^d$ we denote $\Lambda^p := \{ \lambda \in \Lambda \mid d(\lambda) = p \}$, and when $\lambda \in \Lambda$ and $F\subseteq \Lambda$, we denote $\lambda F := \{  \lambda \mu \mid \mu\in F \mbox{ with }\ s(\lambda) = r(\mu)  \} $ and\break $F \lambda := \{ \mu \lambda \ | \mu\in F \mbox{ with }\ s(\mu)=r(\lambda) \}$. We also write $\Lambda^1 := \Lambda^0 \cup \bigcup_{i=1}^d\Lambda^{e_i}$.

For $\mu, \nu \in \Lambda$ let
\[
\Lambda^{\textup{min}}(\mu, \nu) := \{(\alpha, \beta) : \mu \alpha = \nu \beta, d(\mu \alpha) = d(\mu) \vee d(\nu) \}
\]
be the set of minimal common extenders of $\mu$ and $\nu$. We will say that $(\Lambda,d)$ is finitely aligned if $|\Lambda^{\textup{min}}(\mu, \nu)|<\infty$ for every $\mu,\nu \in \Lambda$. Given a vertex $v\in \Lambda^0$, we say that a subset $F\subseteq v \Lambda$ is \emph{exhaustive} if for every $\mu \in v\Lambda$ there is $\nu \in F$ such that $\Lambda^{\textup{min}}(\mu, \nu) \neq \emptyset$. Given a finitely aligned higher rank graph $(\Lambda,d)$, a set of partial isometries $S = \{S_{\lambda}\}_{\lambda \in \Lambda}$ is called a \emph{Toeplitz-Cuntz-Krieger} $\Lambda$-family if
\begin{enumerate}
\item[(P)]
$\{S_v\}_{v\in \Lambda^0}$ is a collection of pairwise orthogonal projections;
\item[(C)]
$S_{\mu} S_{\nu} = S_{\mu \nu}$ when $s(\mu) = r(\nu)$;
\item[(NC)]
$S_{\mu}^*S_{\nu} = \sum_{(\alpha,\beta)\in \Lambda^{\textup{min}}(\mu, \nu)}S_{\alpha}S_{\beta}^*$.
\end{enumerate}
It is called a \emph{Cuntz-Krieger} $\Lambda$-family if it additionally satisfies
\begin{enumerate}
\item[(CK)]
$\prod_{\lambda \in F}(S_v - S_{\lambda}S_{\lambda}^*) = 0$ for every $v\in \Lambda^0$ and all non-empty finite exhaustive sets $F\subseteq v\Lambda$.
\end{enumerate}
We will denote the universal $\ca$-algebra generated by a Cuntz-Krieger $\Lambda$-family by $\ca(\Lambda)$.

Every higher rank graph $(\Lambda,d)$ has a natural product system $X(\Lambda)$ associated to it as in \cite{RS03}. More precisely, for each $p \in \bN^d$ we put a pre-Hilbert $c_0(\Lambda^0)$-bimodule structure on $C_c(\Lambda^p)$ via the formulas
$$
\lip \xi, \eta \rip (v) := \sum_{s(\lambda) = v}\overline{\xi(\lambda)}\eta(\lambda) \ , \ \text{and} \ (a\cdot \xi \cdot b)(\lambda) := a(r(\lambda))\xi(\lambda)b(s(\lambda))
$$
for $\xi,\eta \in X_p$, $a,b \in c_0(\Lambda^0)$, $\lambda \in \Lambda^p$ and $v\in \Lambda^0$.
The separated completion $X(\Lambda)_{p}$ of this pre-Hilbert bimodule then becomes a product system where the identification $X(\Lambda)_{p} \otimes X(\Lambda)_{q} \cong X(\Lambda)_{p+q}$ is given by the map $\chi_{\mu} \otimes \chi_{\nu} \mapsto \delta_{s(\mu),r(\nu)} \cdot \chi_{\mu \nu}$ when $s(\mu) = r(\nu)$ for $\mu \in X_p$ and $\nu \in X_q$. It was shown in \cite[Theorem 5.4]{RS03} that $X(\Lambda)$ is compactly aligned if and only if $\Lambda$ is finitely aligned. Furthermore, $X(\Lambda)$ is regular if and only if $\Lambda$ is row-finite and sourceless (in each color separately).

Given an isometric representation $\psi$ of $X(\Lambda)$, the family $\{ \ t_{\lambda} \ | \ \lambda \in \Lambda \ \}$ given by $t_{\lambda} = \psi_p(\delta_{\lambda})$ becomes a family of partial isometries that satisfies conditions (P), (C) and $t_{\lambda}^*t_{\lambda} = t_{s(\lambda)}$ for any $\lambda \in \Lambda$. We call such a family of operators a $\Lambda$-family. The representation $\psi$ is then Nica-covariant if and only if $\{t_{\lambda}\}_{\lambda \in \Lambda}$ satisfies condition (NC). In fact, Nica-covariant representations $\psi$ of $X(\Lambda)$ are in bijection with Toeplitz-Cuntz-Krieger $\Lambda$ families. We will denote by $\cT(\Lambda):= \cN \cT_{X(\Lambda)}$ the universal C*-algebra generated by a Toeplitz-Cuntz-Krieger family. By \cite[Theorem 5.4]{SY11} we see that the Cuntz-Nica-Pimsner algebra $\cN \cO_{X(\Lambda)}$ coincides with the universal C*-algebra $\ca(\Lambda)$ generated by a Cuntz-Krieger $\Lambda$-family, and that CNP representations of $X(\Lambda)$ are in bijection with Cuntz-Krieger $\Lambda$-families. Denote by $\cT_+(\Lambda)$ the norm closed algebra generated by a universal Toeplitz-Cuntz-Krieger $\Lambda$-family, which coincides with $\cN \cT^+_{X(\Lambda)}$. As a corollary of Theorem \ref{T:CNP-envelope} we obtain the following generalization of \cite[Theorem 3.6]{KK06a} to finitely-aligned higher rank graphs.

\begin{corollary} \label{C:fin-aligned}
Let $\Lambda$ be a finitely aligned higher rank graph. The C*-envelope of $\cT_+(\Lambda)$ coincides with the universal Cuntz-Krieger algebra $\ca(\Lambda)$ associated to $\Lambda$.
\end{corollary}

In Theorem \ref{T:unitary-dil}, the assumption of faithfulness of left actions $\phi_p$ for a product system $X$ over $(G,P)$ is easily seen to be a necessary assumption for every isometric Nica-covariant representation to have an isometric and fully-coisometric dilation. Less clear is the necessity of the assumption that $\phi_p(\cA) \subseteq \cK(X_p)$. However, in \cite[Example 5.5]{SZ08} it is shown that an isometric representation of a product system over $\bN$ fails to have an isometric and fully-coisometric dilation, even with faithful left actions. Our next goal is to show that in the multivariable context, the assumption of double commutation in Corollary \ref{C:unit-dil-SZ} cannot be dropped completely.

\begin{example} \label{E:Nica-needed}
There exists an isometric representation $\psi$ of a regular product system $X(\Lambda)$ coming from a finite sourceless graph $\Lambda$ with no isometric and fully coisometric dilation (and in particular this representation is not Nica-covariant).

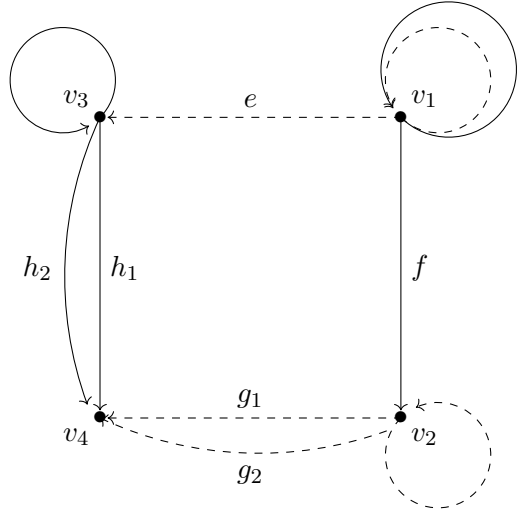
\begin{figure}[hbt]
\begin{center}
\begin{tikzpicture}

\draw[<-, dashed]  
	(0.1,4) -- (4,4);
\draw[<-]  
	(0,0.1) -- (0,4);
\draw[<-] 
	(4,0.1) -- (4,4);
	
\draw[<-, dashed] 
	(0.1,0) -- (4,0);
	
\draw[->] 
	(0,4) arc (-45:300:0.7);
	
\draw[->, dashed] 
	(4,0) arc (135:475:0.7);
	
\draw[->, dashed] 
	(4,4) arc (225:570:0.7);
	
\draw[->] 
	(4,4) arc (225:575:0.9);
	
\draw[->]
	(0,4) arc (155:200:5);
	
\draw[<-,dashed]
	(0,0) arc (245:290:5);
	

\node at (0,0) {$\bullet$};
\node [below left] at (0,0) {$v_4$};
\node at (4,0) {$\bullet$};
\node [below right] at (4,0) {$v_2$};
\node at (4,4) {$\bullet$};
\node [above right] at (4,4) {$v_1$};
\node at (0,4) {$\bullet$};
\node [above left] at (0,4) {$v_3$};

\node [above] at (2,4) {$e$};
\node [right] at (4,2) {$f$};
\node [above] at (2,0) {$g_1$};
\node [below] at (2,-0.5) {$g_2$};
\node [left] at (-0.5,2) {$h_2$};
\node [right] at (0,2) {$h_1$};

\end{tikzpicture}
\end{center}
\end{figure}

We begin by describing a colored graph $G = (V,E)$ on four vertices $V:= \{v_1,v_2,v_3,v_4\}$. Around $v_1$ there are two loops $\ell_1^{(1)}$ and $\ell_1^{(2)}$ of distinct colors $1$ and $2$ respectively, around $v_2$ there is one loop $\ell_2$ of color $2$ and around $v_3$ there is one loop $\ell_3$ of color $1$. There are two edges $h_1,h_2$ from $v_3$ to $v_4$ with color $1$, two edges $g_1,g_2$ from $v_2$ to $v_4$ of color $2$, an edge $f$ from $v_1$ to $v_2$ of color $1$ and an edge $e$ from $v_1$ to $v_3$ of color $2$. This defines a $2$-colored graph as in the drawing. We obtain a higher-rank graph $\Lambda$ on it by specifying the commutation of any pair of concatenating edges of distinct colors
\begin{align} \label{eq:relations}
h_je = g_jf, \ \ \ell_3e = e \ell_1^{(1)}, \ \ \ell_2f = f \ell_1^{(2)}, \ \ \ell_1^{(1)}\ell_1^{(2)} = \ell_2^{(2)}\ell_1^{(1)}.
\end{align}

Our next goal is to define a $\Lambda$-family $\{S_{\lambda}\}$ for $\Lambda$. Let $\cM:= \cH^{(c)}_{1}\oplus \cH^{(c)}_{2} \oplus \cH^{(s)}_{1} \oplus \cH^{(s)}_{2}$ where $\cH_j^{(c)}$ and $\cH_j^{(s)}$ have orthonormal bases $\{\xi_n^{(j)}\}_{n \in \bN}$ and $\{\eta_n^{(j)}\}_{n\in \bN}$ for $j=1,2$ respectively. We first define operators on $\cM$.

Let $V : \cM \rightarrow \cH_1^{(c)} \oplus \cH_2^{(c)}$ be some isometry, and $T_j : \cM \rightarrow \cH_j^{(c)} \oplus \cH_j^{(s)}$ be unitaries for $j=1,2$. Define also a switching operator $W : \cM \rightarrow \cM$ of $\cH_1^{(s)}$ and $\cH_2^{(s)}$ by setting $W(\xi_n^{(j)}) = \xi_n^{(j)}$, $W(\eta_n^{(1)}) = \eta_n^{(2)}$ and $W(\eta_n^{(2)}) = \eta_n^{(1)}$.

We then let $\cH_{v_i}:=\cM$ be a $v_i$-th copy of $\cM$ in $\cH := \bigoplus_{i=1}^4 \cH_{v_i}$ identified via the co-isometry $J_{v_i}: \cH_{v_i} \rightarrow \cM$. We now define a $\Lambda$-family on $\cH$ as follows. The operator $S_{v_i}$ is the orthogonal projection onto $\cH_{v_i}$ for each $i=1,2,3,4$, loop operators are defined via $S_{\ell_1^{(j)}} = S_{v_1}$ for $j=1,2$, $S_{\ell_2} = S_{v_2}$ and $S_{\ell_3} = S_{v_3}$. Finally, non-loop edge operators are given by
$$
S_e = J_{v_3}^*V J_{v_1}, \ \ S_f = J_{v_2}^*VJ_{v_1} \ \ \text{and} ;
$$
$$
S_{h_j} = J_{v_4}^* T_j J_{v_3}, \ \ S_{g_j} = J_{v_4}^*WT_jJ_{v_2}
$$
for $j=1,2$. Then we get by definition of our operators that the relations in equation \eqref{eq:relations} are satisfied at the level of operators. Hence, it follows that the above family of operators extends to a $\Lambda$-family $S:=\{S_{\lambda}\}$ for the $2$-graph $\Lambda$.

However, this $\Lambda$-family fails condition (NC). As $g_1$ and $h_2$ have no minimal common extenders, it will suffice to show $S_{g_1}^*S_{h_2} \neq 0$ to see that (NC) fails. So we compute,
$$
S_{g_1}^*S_{h_2} = J_{v_2}T_1^*WJ_{v_4} J_{v_4}^*T_2J_{v_3} = J_{v_2}T_1^* W T_2 J_{v_3}.
$$
Thus, as $T_1^* W T_2$ is a unitary on $\cM$ and $J_{v_2}, J_{v_3}$ are co-isometries, we clearly see that $S_{g_1}^*S_{h_2} \neq 0$. Thus, $\{S_{\lambda}\}$ fails (NC).

We next show that any dilation of $S$ also fails (NC). Assume towards contradiction that $W=\{W_{\lambda}\}_{\lambda \in \Lambda}$ is a TCK $\Lambda$-family that dilates $S$. Since both $S$ and $W$ are $\Lambda$-families, they give rise to representations of the universal non-selfadjoint algebra generated by $\Lambda$-families. Hence, we may apply Sarason's theorem \cite[Exercise 7.6]{Pau02} to get that for any $\lambda \in \Lambda$ we have
$$
W_{\lambda} = \begin{bmatrix}
* & X_{\lambda} & * \\
0 & S_{\lambda} & * \\
0 & 0 & *
\end{bmatrix}
$$
By taking the $(2,2)$-corners of $W_{s(\lambda)} = W_{\lambda}^*W_{\lambda}$ we see that $S_{s(\lambda)} = X_{\lambda}^*X_{\lambda} + S_{\lambda}^*S_{\lambda}$ so that in fact $X_{\lambda} = 0$ for all $\lambda \in \Lambda$.

As $W$ satisfies (NC), and as $g_1$ and $h_2$ have no minimal extenders, we must have that $W_{g_1}^* W_{h_2} = 0$. By taking the $(2,2)$ corners of this equality we obtain that
$$
0 = X_{g_1}^*X_{h_2} + S_{g_1}^*S_{h_2} = S_{g_1}^*S_{h_2} \neq 0
$$
which is a contradiction. 

Hence, if $\psi$ is the representation of $X$ associate to the $\Lambda$-family $S$, $\psi$ fails to have an isometric and fully-coisometric dilation. Indeed, such a dilation would be Cuntz-Pimsner covariant in the sense that $\psi^{(p)} \circ \phi_p = \psi_e$ so that by \cite[Proposition 5.4]{Fow02} it would automatically be Nica-covariant. This will force the family $W$ associated to this dilation to satisfy (NC), while $W$ dilates $S$ in the above sense.
\end{example}

There is a class of product systems that generalizes those arising from higher-rank graphs as well as algebras arising from $\bN^N$-dynamical systems as in \cite{DFK14}. To describe this class, one needs to introduce the class of topological higher-rank graphs. We will not do this here but instead we direct the reader to the papers of Yeend~\cite{Yee06, Yee07} for the pertinent definitions and additional details. The following is an immediate corollary of Theorem~\ref{T:CNP-envelope} and \cite[Theorem 5.20]{CLSV11}.

\begin{corollary}
Let $\Lambda$ be a compactly aligned topological $k$-graph. Let $\G_{\Lambda}$ be Yeend's boundary path groupoid~\cite[Definition 4.8]{Yee07} and let $X(\Lambda)$ be the product system associated with $\Lambda$ as in \cite[Proposition 5.9]{CLSV11}. Then the C*-envelope of $\cN \cT^+_{X(\Lambda)}$ coincides with the groupoid $\ca$ algebra $\ca(\G_{\Lambda})$.
\end{corollary} 

In the work of Davidson, Fuller and Kakariadis \cite{DFK14}, non-commutative C*-dynamical systems $\alpha : \bN^N \rightarrow \End(A)$ and the associated Nica-Toeplitz $\N \T(A,\alpha)$ and Cuntz-Nica-Pimsner $\N \O(A,\alpha)$ crossed product algebras were considered. The majority of \cite[Subsection 4.3]{DFK14} is devoted to obtaining a multivariable tail-adding technique that allows them to prove a strong Morita equivalence between $\N \O(A,\alpha)$ and $\N \O(B, \beta)$ where $\beta : \bN^N \rightarrow \End(B)$ is \emph{injective}. This yields several consequences for non-injective systems, among which are a gauge-invariant uniqueness theorem \cite[Theorem 4.3.17]{DFK14} and a characterization of the C*-envelope \cite[Corollary 4.3.18]{DFK14}. In a recent paper of the first author with Kakariadis \cite{DK+}, one can recover \cite[Theorem 4.3.17]{DFK14} in an alternative way, and Theorem~\ref{T:CNP-envelope} can then be used to obtain \cite[Corollary 4.3.18]{DFK14} for non-degenerate systems as an immediate consequence. This makes for a substantial simplification of the characterization of the C*-envelope, as it no longer requires the use of a multivariable tail-adding technique.

\section{Hao-Ng isomorphism for product systems} \label{S:HNI}

In this section we examine the Hao-Ng isomorphism problem in the context of product systems. In what follows, we assume some familiarity with the theory of reduced crossed products of (not-necessarily-selfadjoint) operator algebras, as it appears in \cite{KR16}.  

Let $(G,P)$ be a quasi-lattice ordered group and $X=(X_p)_{p \in P}$ a compactly aligned product system over $P$ with coefficients in $\cA$. We henceforth regard each $X_p$ as an operator subspace of $\N \T_X$ via its representation on Fock space. Consider an action $\alpha: \G \rightarrow \Aut \N\T_X$ so that $\alpha_{g}(X_p)= X_p$, for all $g \in \G$ and $p \in P$. We call such an action $\alpha$ a \textit{generalized gauge action} and we say that the group $\G$ \textit{acts on }$X$. Clearly the action $\alpha$ restricts to a generalized gauge action $\alpha : \G \rightarrow \Aut  (\N\T^+_X )$. This generalized gauge action extends to a generalized gauge action on $\ca_e(\N \T^+_X)$ so that when $G$ is abelian, we have a generalized gauge action on $\ca_e(\N \T^+_X) \cong \N\O_X^r$ by Theorem \ref{T:CNP-envelope}. The crossed product of $\N\O_X^r$ by such actions plays an important role in $\ca$-algebra theory: in the case of a Cuntz or a Cuntz-Krieger $\ca$-algebra, examples of such actions are the so-called \textit{quasi-free} actions whose crossed products have been studied extensively \cite{EF, KK, LN}.

For each $p \in P$, let $X_p\cpr$ be the closed subspace of $\N\T_X\cpr$ generated by $C_c(\G , X_p) \subseteq \N\T_X\cpr$, i.e., all finite sums of the form $\sum_s x_s U_s$, $x_s \in X_p$, $s \in \G$. Just as in \cite[Lemma 7.11]{KR16}, one can verify that for every $p \in P$ we have that $X_p\cpr$ is an $\A\cpr$-correspondence with inner product defined by $\lip x,y\rip =x^*y$ for $x , y \in X_p \cpr\subseteq \N\T_X\cpr$. Furthermore, it is easily seen that $(X_p \cpr)(X_q \cpr) \subseteq X_{pq} \cpr$ is dense for any $p,q \in P$. Therefore $( X_p\cpr )_{p \in P}$ forms a product system that we denote as $X \cpr$. 

The $\ca$-algebra $\N\T_X \cpr$ also contains a non-selfadjoint crossed product algebra which we denote as $\N\T^+_X\cpr$, which is the norm closed algebra generated by $x_sU_s$ where $x_s \in X_p$ and $s\in \G$ for $p\in P$. In particular, we have the inclusions
\[
X\cpr \subseteq \N\T^+_X\cpr\subseteq \N\T_X \cpr.
\]
Hence, $\N\T^+_X\cpr$ is generated as a closed algebra by $X \cpr$ and $\N\T_X \cpr$ is a $\ca$-cover for $\N\T^+_X\cpr$.

\begin{remark}\label{R:clarify copy}
(i) The reader familiar with the theory of non-selfadjoint crossed products (as developed in \cite{KR16}) recognizes that $\N\T^+_X\cpr$ coincides with the reduced crossed product of the dynamical system $( \N\T_X^+, \G, \alpha)$, as defined in \cite[Definition 3.17]{KR16}. Indeed this follows from an immediate application of \cite[Corollary 3.15]{KR16}.

(ii) When $G$ is abelian, we see that $\N\T_X^+\cpr$ is completely isometrically isomorphic to the natural subalgebra of $\N\O^r_X \cpr$ generated by $x_sU_s$ where $x_s \in X_p$ and $s\in \G$ for $p\in P$. This follows from Theorem \ref{T:CNP-envelope} and \cite[Corollary 3.16]{KR16}. Hence, we also obtain an injective copy of $X\cpr$ sitting naturally inside $\N\O^r_X\cpr$.
\end{remark} 
 
In order to apply the theory of Section~\ref{Sec:envelope} we need to verify that compact alignment of a product system is preserved when taking crossed products by a generalized gauge action of a discrete group. We will use the following well-known fact.
 
 \begin{lemma} \label{basicKatsura}
 Let $X$ be a $\ca$-correspondence and $\psi: X\rightarrow B(\H)$ a representation on a Hilbert space $\H$. If $\psi$ is injective, then the $*$-homomorphism which extends the association
 \[
\K(X)\ni \theta_{x,y}\longmapsto \psi(x)\psi(y)^*\in B(\H)
\]
is also injective.
 \end{lemma}
 
 \begin{proposition}
Let $(G,P)$ be an abelian, lattice ordered group and $X$ a product system over $P$. Let $\alpha : \G \rightarrow \Aut ( \N\T_X )$ be a generalized gauge action by a discrete group $\G$. If $X$ is compactly aligned, then $X \cpr$ is also compactly aligned.
\end{proposition}
 
 \begin{proof}
 Let $s, t \in P$. Consider arbitrary $x',y' \in X_s$, $\hat{x}', \hat{y}' \in X_t$ and $g', \hat{g}', h, \hat{h}'\in \G$ and let 
 \[
 S:= \theta_{x'U_{g'}, y'U_{h'}}\in \K(X_s \rtimes_{\alpha}^r\G) \mbox{ and } T:= \theta_{\hat{x}'U_{\hat{g}'}, \hat{y}'U_{\hat{h}'}}\in \K(X_t \rtimes_{\alpha}^r\G). 
 \]
If we verify that for such special compact operators $S$ and $T$, the operator $(S\otimes I)(T\otimes I) \in \L(X_{s\vee t} \rtimes_{\alpha}^r\G)$ is compact, then the compactness of $(S\otimes I)(T\otimes I)$ for arbitrary compact operators $S\in \K(X_s \rtimes_{\alpha}^r\G)$ and $T\in \K(X_t \rtimes_{\alpha}^r\G)$ will follow from an easy density argument.

Thus, let $z \in X_{s\vee t}$ and $h \in \G$. Then it is easy to see that there exist $x, y \in X_s$ and $\hat{x}, \hat{y} \in X_t$ so that 
\begin{equation} \label{easycalign}
(S\otimes I)(T\otimes I) (zU_h)= U_g\big((\theta_{x,y}\otimes I)(\theta_{\hat{x}, \hat{y}} \otimes I)z \big)U_h.
\end{equation}
Since $X$ is compactly aligned, 
\begin{equation} \label{thetas}
(\theta_{x,y}\otimes I)(\theta_{\hat{x}, \hat{y}} \otimes I) = \lim_n S_n ,
\end{equation}
where $S_n := \sum_i \, \theta_{x_{n,i}, y_{n,i}}$,$n\in \bbN$, are generalized finite rank operators with $x_{n,i}, y_{n,i} \in X_{s\vee t}$, for all $i$ and $n\in \bbN$. For each $n \in \bbN$ we define
\[
R_n := \sum_{i} \theta_{\alpha_g(x_{n,i})U_g, \,\,y_{n,i}I}\in \K(X_{s\vee t} \rtimes_{\alpha}^r\G).
\]
For a ``rank one" operator $\theta_{z,w}$ with $z,w\in X_p$ (or $z,w \in X \rtimes_{\alpha}^r \G$), Lemma \ref{basicKatsura} allows us to identify $\theta_{z,w}$ with a product of operators $zw^* := l_p(z)l_p(w)^*$ of the left regular Fock representation. We use this to show that the sequence $\{R_n\}_{n=1}^{\infty}$ is bounded. Indeed,
\begin{align*}
\|R_n\|&= \| \sum_i\alpha_g(x_{n,i})U_gy_{n,i}^*\|  = \|U_g\sum_i \, x_{n,i}y_{n,i}^*\| \\
           &=\|\sum_i \, x_{n,i}y_{n,i}^*\| = \|\sum_i \, \theta_{x_{n,i}, y_{n,i}}\| \\
           &=\|S_n\|,
\end{align*}
As $\{S_n\}_{n=1}^{\infty}$ is convergent, we see that $\{R_n\}_{n=1}^{\infty}$ is bounded. We next claim that $\textsc{sot--}\lim_nR_n= (S\otimes I)(T\otimes I)$.

Indeed, since $\{R_n\}_{n=1}^{\infty}$ is bounded, it suffices to verify the claim on a suitable dense subset of $X_{s \vee t}$. Let $z \in X_{s \vee t}$ and $ h \in \G$. Then by \eqref{thetas},
\begin{align*}
\lim_nR_n(zU_h) &=\lim_n \sum_i \, \alpha_g(x_{n,i})U_gy_{n,i}^*zU_h =\lim_n U_g\big(\sum_i \, x_{n,i}y_{n,i}^*z\big)U_h \\
		&= \lim_n\, U_g(S_nz)U_h = U_g\big((\theta_{x,y}\otimes I)(\theta_{\hat{x}, \hat{y}} \otimes I)z \big)U_h. 
\end{align*}
However, by \eqref{easycalign}, 
$$
U_g\big((\theta_{x,y}\otimes I)(\theta_{\hat{x}, \hat{y}} \otimes I)z \big)U_h =  (S\otimes I)(T\otimes I)(zU_h)
$$ 
and the claim follows.

In light of the previous claim, to finish the proof it will suffice to show that $(S\otimes I)(T\otimes I)$ is compact. Thus, assume $\{R_n\}_{n=1}^{\infty}$ is Cauchy sequence, and examine the norm of $R_n-R_m$ for $m,n \in \bbN$, on a dense subset of the unit ball of $X_{s \vee t}$. Indeed, by Lemma~\ref{basicKatsura},
\begin{align*}
\big\|(R_n-R_m)(\sum_j \, z_jU_{h_j})\big\|&=\big\|U_g\big(\sum_{i,j} x_{n, i}y_{n,i}^*z_jU_{h_j} -\sum_{k,j} x_{m, k}y_{m,k}^*z_jU_{h_j}\big)\| \\
&=\big\|\sum_{i,j} x_{n, i}y_{n,i}^*z_jU_{h_j} -\sum_{k,j} x_{m, k}y_{m,k}^*z_jU_{h_j}\big\|\\
&=\big\|\big(\sum_i\, x_{n, i}y_{n,i}^*-\sum_k \, x_{m,k}y_{m,k}^*\big) 
\big(\sum_j \, z_jU_{h_j}\big)\big\| \\
&\leq\big\| \sum_i\, x_{n, i}y_{n,i}^*-\sum_k \, x_{m,k}y_{m,k}^*\big\|\big\| \sum_j \, z_jU_{h_j}\big\| \\
&\leq \big\|(S_n-S_m)\big\|.
\end{align*}
The proof is thus concluded, and $X \rtimes_{\alpha}^r \G$ is compactly aligned.
 \end{proof} 
 
Now we need  a good supply of Nica-covariant representations for the crossed product system $X\cpr$.
 \begin{lemma} \label{L:supply}
 Let $(G,P)$ be an abelian, lattice ordered group and $X$ a product system over $P$. Let $\alpha : \G \rightarrow \Aut ( \N\T_X )$ be a generalized gauge action by a discrete group $\G$. If $\pi \colon \cN\cT_X \rightarrow B(\H)$ is a $*$-representation, then the restriction of the regular representation 
 \[
 \Ind_{\pi} \colon \cN\cT_X \cpr \rightarrow B(\H \otimes \ell^2(\G))
 \]
 on the product system $X\cpr \subseteq \cN\cT_X \cpr$ forms a Nica-covariant representation of $X\cpr$.
 \end{lemma}
 
 \begin{proof}
Let $p \in P$. Since the correspondence $X\cpr$ is non-degenerate, we have that $\Ind_{\pi}^{(p)}(I)$ projects on the subspace
\begin{align*}
\Ind_{\pi}\big( (X\cpr)_p\big)(\H \otimes \ell^2(\G)) &= \Ind_{\pi} (X_p\cpr)(\H \otimes \ell^2(\G)) \\
									&= \pi(X_p)(\H)  \otimes \ell^2(\G)
\end{align*}
and so $\Ind_{\pi}^{(p)}(I) = \pi^{(p)}(I) \otimes I$. That $\Ind_{\pi}$ yields a Nica-covariant isometric representation of $X\cpr$ is then easily verified. 
\end{proof}
 
\begin{theorem} \label{HNtensor2}
Let $(G,P)$ be an abelian, lattice ordered group and $X$ a compactly aligned product system over $P$ with coefficients in $\A$. Let\break $\alpha : \G \rightarrow \Aut  ( \N\T_X )$ be a generalized gauge action by a discrete group $\G$. Then
\[
\cN\cT^{+}_X \cpr \simeq \cN\cT^+_{X\cpr}.
\]
Therefore, we have $$\cenv\big (\cN\cT^+_{X} \cpr \big) \simeq  \N\O^r_{X\cpr}.$$
\end{theorem}

\begin{proof} Let $i=i_X \colon X \rightarrow B(\H)$ be the universal Nica-covariant representation of $X$ and let 
\[
V\colon P \longrightarrow B(\ell^2(P)); p \longmapsto V_p
\]
be the left regular representation of $P$ on $\ell^2(P)$. Let $\pi := i\otimes V$ be defined as in \eqref{eq:tensor rep}. By Lemma~\ref{L:Nicatensor} $\pi$ is a Nica-covariant representation of $X$. It is easily seen to satisfy the requirements of Fowler's Theorem \ref{T:toep-uniq} and therefore it induces a faithful representation $\pi_* : \cN\cT_X \rightarrow B(\H \otimes \ell^2(P))$.

Consider the induced regular representation 
\[
\Ind_{\pi_*}\colon \cN\cT_X \cpr \longrightarrow B\big( \H \otimes \ell^2(P)\otimes \ell^2(\G)\big)
\]
and let 
\[
\psi\colon X \cpr \longrightarrow B\big( \H \otimes \ell^2(P)\otimes \ell^2(\G)\big)
\]
be the restriction of $\Ind_{\pi_*}$ to $X \cpr \subseteq \N\T \cpr$. By Lemma~\ref{L:supply}, $\psi$ is Nica-covariant. 

We claim that $\psi$ satisfies the requirements of Fowler's Theorem for $X \cpr$. Indeed let $Q_p\in \ell^2(P)$ be the projection on the one dimensional subspace corresponding to the characteristic function of $p \in P$ and let 
$$\hat{Q}_p = I\otimes Q_{p}\otimes I \in B\big( \H \otimes \ell^2(P)\otimes \ell^2(\G)\big).
$$ 
If $f \in C_c(\G, X_p)$ is a finite sum in $X_p\cpr$ for $ p \in P$,  and $h \in  \H \otimes \ell^2(P)\otimes \ell^2(\G)$, then
\begin{align*}
\big(\hat{Q}_e \psi_p(f)h\big)(s) &= \hat{Q}_e (\Ind_{\pi_*} f)h(s) \\ &= \sum_{s \in \G} (I\otimes Q_{e}) \pi_*\big(\alpha_{t}^{-1}(f(s))\big)h(s^{-1}t)\\
&= \sum_{s \in \G} (I\otimes Q_{e})\big(i_p(\alpha_{t}^{-1}(f(s))) \otimes V_p\big)h(s^{-1}t)\\
&=\sum_{s\in \G} \big(i_p(\alpha_{t}^{-1}(f(s))) \otimes Q_eV_p\big)h(s^{-1}t) = 0
\end{align*}
Hence $\hat{Q}_e  \psi^{(p)}(I)=0$ and so each projection $I -\psi^{(p)}(I)$, $p \in P$, dominates $\hat{Q}_e$. Therefore if $F\subseteq P\backslash \{e\}$ is any finite set, then the product $ \prod_{p\in F} \big(I -\psi^{(p)}(I)\big)$ also dominates $\hat{Q}_e$. Hence
\begin{equation} \label{normeq}
\big\|\psi_e(f) \prod_{p\in F} \big(I -\psi^{(p)}(I)\big)\big\| \geq\big\| \psi_e(f) \hat{Q}_e\big\| =\big\|\Ind_{\pi_*}(f)\hat{Q}_e \big\|, 
\end{equation}
for any $f \in \C_c(\G, \A)$.
On the other hand 
$$\pi_* \mid_{\A} \simeq \big(\oplus_{p \in P}(I \otimes Q_p) \pi_* \big) \mid_{\A} \simeq \big(\oplus (I \otimes Q_e) \pi_* \big)\mid_{\A}, $$ 
i.e., the restriction of $\pi_*$ on $\A$ is unitarily equivalent to a direct sum indexed by $P$ of copies of $(I \otimes Q_e) \pi_* $ restricted to $\A$. From this we obtain,
\[
\psi_e = {\Ind_{\pi_*}}\mid _{\A \cpr} \, \simeq \oplus {\Ind_{ Q_e\pi_*}}\mid _{\A \cpr} \simeq \oplus {\hat{Q}_e\Ind_{\pi_* }}\mid _{\A \cpr}. 
\]
Combining the above with (\ref{normeq}) we now obtain 
\[
\big\|\psi_e(f) \prod_{p\in F} \big(I -\psi^{(p)}(I)\big)\big\| \geq \big\|\Ind_{\pi_*}(f)\hat{Q}_e \big\| = \|\psi_e(f) \|
\]
for any $f \in \C_c(\G, \A)$, which establishes that $\psi$ satisfies the requirements of Fowler's Theorem for $X \cpr$.

From this we obtain that the induced representation $\psi_{*}$ is a faithful $*$-representation of $\N\T_{X\cpr}$. In addition, note that $\psi_{*}(\N\T^+_{X\cpr}) \simeq \N\T^+_{X\cpr}$ is equal to the closed linear span of
\[
\psi_{*}(X\cpr)=\bigcup_{p \in P} \psi_p(X_p\cpr)=\bigcup_{p \in P}\overline{\Ind_{\pi_*}C_c(\G, X_p)}.
\]
However, $\N\T^+_X\cpr$ is also completely isometrically isomorphic to the closed linear span of  $$\bigcup_{p \in P}\overline{\Ind_{\pi_*}C_c(\G, X_p)}.$$ Hence, $\cN\cT^{+}_X \cpr \simeq \cN\cT^+_{X\cpr}$. Finally by Theorem~\ref{T:CNP-envelope} we see that
$$
\cenv\big (\cN\cT^+_X \cpr \big) \simeq  \cenv(\cN\cT^+_{X\cpr}) \simeq \cN\cO^r_{X\cpr}.
$$ \end{proof}

We now use the above to obtain our extension of the Hao-Ng isomorphism. 

\begin{theorem} \label{HaoNgreduced}
Let $(G,P)$ be an abelian, lattice ordered group and let $X$ be a compactly aligned and product system over $P$. Let $\alpha : \G \rightarrow \Aut  ( \N\T_X )$ be a generalized gauge action by a discrete group $\G$. Then
\[
\N\O^r_X \cpr \simeq\N\O^r_{X\cpr}.
\]
  \end{theorem}
  
  \begin{proof} 
  By Theorem~\ref{HNtensor2} we have $$\cenv\big (\cN\cT^+_X \cpr \big) \simeq  \N\O^r_{X\cpr}.$$ On the other hand \cite[Theorem 2.5]{KatsIMRN} implies that $$\cenv\big (\cN\cT^+_X\cpr \big) \simeq \cenv (\cN\cT^+_X) \cpr .$$ Hence $\cenv (\cN\cT^+_X) \cpr  \simeq \N\O^r_{X\cpr}$. Now an application of Theorem~\ref{T:CNP-envelope} shows that $\cenv (\cN\cT^+_X) \simeq \N\O^r_X$ via a $\G$-equivariant map that intertwines the corresponding generalized gauge actions and the conclusion follows.
 \end{proof}

Let us indicate the utility of Theorem~\ref{HaoNgreduced} with a quick application. Let $(\Lambda, d)$ be a finitely aligned $k$-graph and let $\G$ be a discrete group acting on $X(\Lambda)$ via $\alpha: \G \rightarrow \Aut \ca(\Lambda)$, e.g., let $\G=\bbF^k$ be the free group on $k$-generators acting on $\ca(\Lambda)$ by twisting the generators with unimodular scalars. As we discussed in Remark~\ref{R:clarify copy}(ii), $X(\Lambda) \cpr$ is just the closed subspace of $\ca(\Lambda) \cpr$ generated by all monomials of the form $x U_s$, with $x \in X(\Lambda) $ and $s \in \G$. Theorem~\ref{HaoNgreduced} implies that the Cuntz-Nica-Pimsner $\ca$-algebra of $X(\Lambda) \cpr$ is isomorphic to $\ca(\Lambda) \cpr$. Even in very special cases, it is quite intricate to verify this result directly from the definition of the Cuntz-Nica-Pimsner $\ca$-algebra.

\subsection*{Acknowledgments} The present paper grew out of discussions between the authors during the Multivariable Operator Theory Conference, which was held  at the Technion, Israel, June 18-22, 2017. The authors are grateful to the organizers of the conference for the invitation to participate and for their hospitality. The authors wish to thank Ken Davidson for identifying a mistake in Example \ref{E:Nica-needed} of an earlier version of the paper.


\end{document}